\documentclass[12pt,leqno,draft]{article} 
\usepackage{amsmath,amssymb,amsthm,amsfonts,bm}
\usepackage{enumerate,color}
\usepackage{indentfirst}
\makeatletter
\def\@cite#1#2{[{{\bfseries #1}\if@tempswa , #2\fi}]}
\renewcommand{\section}{%
\@startsection{section}{1}{\z@}
{0.5truecm plus -1ex minus -.2ex}%
{1.0ex plus .2ex}{\bfseries\large}}
\def\@seccntformat#1{\csname the#1\endcsname.\ }
\makeatother

\numberwithin{equation}{section} 
\newtheorem{thm}{Theorem}[section]

\newtheorem{lem}[thm]{Lemma}

\theoremstyle{definition}
\newtheorem{df}{Definition}[section]
\newtheorem{remark}{Remark}[section]

\newtheorem*{prth1.1}{Proof of Theorem 1.1}
\newtheorem*{prth1.2}{Proof of Theorem 1.2}

\newcommand{\ep}{\varepsilon}

\newcommand{\RN}{\mathbb{R}^N}

\begin{document}
\footnote[0]
    {2010{\it Mathematics Subject Classification}\/. 
    Primary: 35K35, 35K52, 35K59. 
    }
\footnote[0]
    {{\it Key words and phrases}\/: 
    Cahn--Hilliard systems; tumor growth; asymptotic analysis; 
    unbounded domains.  
    }
\begin{center}
    \Large{{\bf 
Asymptotic analysis for Cahn--Hilliard type \\
phase field systems related to tumor growth \\
in general domains 
           }}
\end{center}
\vspace{5pt}
\begin{center}
    Shunsuke Kurima\footnote{Partially supported 
    by JSPS Research Fellowships for Young Scientists (No.\ 18J21006).} \\
    \vspace{2pt}
    Department of Mathematics, 
    Tokyo University of Science\\
    1-3, Kagurazaka, Shinjuku-ku, Tokyo 162-8601, Japan\\
    {\tt shunsuke.kurima@gmail.com}\\
    \vspace{2pt}
\end{center}
\begin{center}    
    \small \today
\end{center}

\vspace{2pt}
\newenvironment{summary}
{\vspace{.5\baselineskip}\begin{list}{}{%
     \setlength{\baselineskip}{0.85\baselineskip}
     \setlength{\topsep}{0pt}
     \setlength{\leftmargin}{12mm}
     \setlength{\rightmargin}{12mm}
     \setlength{\listparindent}{0mm}
     \setlength{\itemindent}{\listparindent}
     \setlength{\parsep}{0pt}
     \item\relax}}{\end{list}\vspace{.5\baselineskip}}
\begin{summary}
{\footnotesize {\bf Abstract.}
This article considers 
a limit system by passing to the limit in 
the following Cahn--Hilliard type phase field system related to tumor growth 
as $\beta\searrow0$: 
 \begin{equation*}
  \begin{cases} 
    \alpha\partial_{t} \mu_{\beta}  
         + \partial_{t} \varphi_{\beta}  
         -\Delta\mu_{\beta} = p(\sigma_{\beta} - \mu_{\beta})  
         & \mbox{in}\ \Omega\times(0, T),
 \\[1mm]
         \mu_{\beta} = \beta\partial_{t} \varphi_{\beta}  
                            +(-\Delta+1)\varphi_{\beta} + \xi_{\beta} 
                            + \pi(\varphi_{\beta}),\ \xi_{\beta} \in B(\varphi_{\beta})
         & \mbox{in}\ \Omega\times(0, T),
 \\[1mm]
         \partial_{t} \sigma_{\beta} 
         -\Delta\sigma_{\beta} = -p(\sigma_{\beta} - \mu_{\beta})   
         & \mbox{in}\ \Omega\times(0, T)  
  \end{cases}
\end{equation*}
in a {\it bounded} or an {\it unbounded} domain $\Omega \subset \RN$   
with smooth bounded boundary. 
Here   
$N\in\mathbb{N}$, $T>0$, $\alpha>0$, $\beta>0$, $p\geq0$, 
$B$ is a maximal monotone graph and $\pi$ is a Lipschitz continuous function. 
In the case that $\Omega$ is a bounded domain, 
$p$ and $-\Delta+1$ 
are replaced with $p(\varphi_{\beta})$ and $-\Delta$, respectively, 
and 
$p$ is a Lipschitz continuous function,     
Colli--Gilardi--Rocca--Sprekels \cite{CGRS-2017} 
have proved existence of solutions to the limit problem 
with this approach 
by applying the Aubin--Lions lemma 
for the compact embedding 
$H^1(\Omega) \hookrightarrow L^2(\Omega)$ 
and the continuous embedding 
$L^2(\Omega) \hookrightarrow (H^1(\Omega))^{*}$.  
However, the Aubin--Lions lemma cannot be applied directly    
when $\Omega$ is an {\it unbounded} domain.  
The present work establishes existence of weak solutions 
to the limit problem 
both  
in the case of {\it bounded} domains and in the case of {\it unbounded} domains.  
To this end 
we construct an applicable theory for both of these two cases 
by noting that the embedding $H^1(\Omega) \hookrightarrow L^2(\Omega)$ 
is not compact in the case that $\Omega$ is an unbounded domain.  
}
\end{summary}
\vspace{10pt}

\newpage
\section{Introduction} \label{Sec1}


\subsection{Background} \label{Subsec1.1} 

The problem 
\begin{equation}\label{P0}\tag{P0}
  \begin{cases} 
    \alpha\partial_{t} \mu + \partial_{t} \varphi 
         -\Delta\mu = p(\varphi)(\sigma - \gamma\mu)  
         & \mbox{in}\ \Omega\times(0, T),
 \\[3mm]
         \mu = \beta\partial_{t}\varphi -\Delta\varphi + G'(\varphi)
         & \mbox{in}\ \Omega\times(0, T),
 \\[3mm]
         \partial_{t} \sigma  
         -\Delta\sigma = -p(\varphi)(\sigma - \gamma\mu)   
         & \mbox{in}\ \Omega\times(0, T),
 \\[3mm]
         \partial_{\nu}\mu = \partial_{\nu}\varphi = \partial_{\nu}\sigma = 0                                   
         & \mbox{on}\ \partial\Omega\times(0, T),
 \\[2.5mm]
        \mu(0) = \mu_0, \varphi(0) = \varphi_{0}, \sigma(0) = \sigma_0                                          
         & \mbox{in}\ \Omega
  \end{cases}
\end{equation}
is a Cahn--Hilliard type phase field system 
related to a tumor growth model which was produced in 
\cite{HZO-2012, HKNZ-2015, WZZ-2014} 
(in \cite{CGH-2015, CGRS-2015, FGR-2015} 
the system was further studied analytically).   
Here $\Omega$ is a three-dimensional bounded domain 
with smooth bounded boundary $\partial\Omega$,  
$\partial_{\nu}$ denotes differentiation with
respect to the outward normal of $\partial\Omega$, 
$p$ is a nonnegative function, 
$G'$ is the first derivative of a nonnegative potential $G$,   
$\alpha>0$, $\beta>0$, $\gamma>0$, 
$T>0$, and
$\mu_0, \varphi_0, \sigma_0$ are given functions.  
The unknown function $\varphi$ is 
an order parameter which can be set as follows: 
\begin{itemize}\setlength{\itemsep}{0cm}
\item $\varphi \simeq 1$ in the tumorous phase. 
\item $\varphi \simeq -1$ in the healthy cell phase. 
\end{itemize}
The unknown function $\mu$ is the related chemical potential 
which is specified by the second equation in \eqref{P0}, 
depending on whether 
$\beta > 0$ (the viscous Cahn--Hilliard equation) 
or 
$\beta=0$ (the Cahn--Hilliard equation) 
(see \cite{CH-1958, ES-1996, EZ-1986}). 
The unknown function $\sigma$ represents the nutrient concentration  typically fulfills as follows: 
\begin{itemize}\setlength{\itemsep}{0cm}
\item $\sigma \simeq 1$ in a nutrient-rich extracellular water phase. 
\item $\sigma \simeq 0$ in a nutrient-poor extracellular water phase. 
\end{itemize}
The function $G_{cl}$ defined by 
\begin{align*}
G_{cl}(r) := \frac{1}{4}(r^2-1)^2 = 
 \frac{1}{4}((r^2-1)^{+})^2 + \frac{1}{4}((1-r^2)^{+})^2 
\quad \mbox{for}\ r\in \mathbb{R}
\end{align*} 
is called the classical double well potential which is a typical example of $G$. 

Recently, 
for a bounded domain $\Omega \subset \mathbb{R}^3$  
Colli--Gilardi--Rocca--Sprekels \cite{CGRS-2017} 
have proved existence of solutions to  
the limit system 
\begin{equation*}
  \begin{cases} 
    \alpha\partial_{t} \mu + \partial_{t} \varphi 
         -\Delta\mu = p(\varphi)(\sigma - \mu)  
         & \mbox{in}\ \Omega\times(0, T),
 \\[3mm]
         \mu = -\Delta\varphi + \xi + \pi(\varphi),\ \xi \in B(\varphi)
         & \mbox{in}\ \Omega\times(0, T),
 \\[3mm]
         \partial_{t} \sigma  
         -\Delta\sigma = -p(\varphi)(\sigma - \mu)   
         & \mbox{in}\ \Omega\times(0, T),
 \\[3mm]
         \partial_{\nu}\mu = \partial_{\nu}\varphi = \partial_{\nu}\sigma = 0                                   
         & \mbox{on}\ \partial\Omega\times(0, T),
 \\[2.5mm]
        \mu(0) = \mu_0, \varphi(0) = \varphi_{0}, \sigma(0) = \sigma_0                                          
         & \mbox{in}\ \Omega
  \end{cases} 
\end{equation*}
by passing to the limit in 
the following Cahn--Hilliard type phase field system  
as $\beta\searrow0$: 
\begin{equation*}
  \begin{cases} 
    \alpha\partial_{t} \mu_{\beta}  
         + \partial_{t} \varphi_{\beta}  
         -\Delta\mu_{\beta} = p(\varphi_{\beta})(\sigma_{\beta} - \mu_{\beta})  
         & \mbox{in}\ \Omega\times(0, T),
 \\[3mm]
         \mu_{\beta} = \beta\partial_{t} \varphi_{\beta}  
                            -\Delta\varphi_{\beta} + \xi_{\beta} 
                            + \pi(\varphi_{\beta}),\ \xi_{\beta} \in B(\varphi_{\beta})
         & \mbox{in}\ \Omega\times(0, T),
 \\[3mm]
         \partial_{t} \sigma_{\beta} 
         -\Delta\sigma_{\beta} = -p(\varphi_{\beta})(\sigma_{\beta} - \mu_{\beta})   
         & \mbox{in}\ \Omega\times(0, T),
 \\[3mm]
         \partial_{\nu}\mu_{\beta} = \partial_{\nu}\varphi_{\beta} 
         = \partial_{\nu}\sigma_{\beta} = 0                                   
         & \mbox{on}\ \partial\Omega\times(0, T),
 \\[2.5mm]
        \mu_{\beta}(0) = \mu_0, \varphi_{\beta}(0) = \varphi_{0}, 
        \sigma_{\beta}(0) = \sigma_0                                          
         & \mbox{in}\ \Omega 
  \end{cases}
\end{equation*}
under the five conditions: 
\begin{enumerate} 
\setlength{\itemsep}{0mm}
 \item[(J1)] $\alpha, \beta \in (0, 1)$.  
 \item[(J2)] The function $p : \mathbb{R} \to \mathbb{R}$ is nonnegative, 
                 bounded and Lipschitz continuous.  
 \item[(J3)] $B \subset \mathbb{R}\times\mathbb{R}$                                
 is a maximal monotone graph with effective domain $D(B)$ 
 and $B=\partial\widehat{B}$, 
 where $\partial\widehat{B}$ denotes the subdifferential of 
 a proper lower semicontinuous convex function 
 $\widehat{B} : \mathbb{R} \to [0, +\infty]$.             
 \item[(J4)] The function $\pi:=\widehat{\pi}'$ is Lipschitz continuous, 
                 where $\widehat{\pi} \in C^1(\mathbb{R})$ is 
                                                             a nonnegative function.  
 \item[(J5)] $\mu_{0}, \sigma_0 \in L^2(\Omega)$, 
 $\varphi_0 \in H^1(\Omega)$  
 and $G(\varphi_{0}) \in L^1(\Omega)$, 
 where $G:=\widehat{B}+\widehat{\pi}$. 
 \end{enumerate} 
In particular, 
they showed 
\begin{align}\label{keyofpre}
\zeta_{\beta}:=\alpha\mu_{\beta} + \varphi_{\beta} 
\to \zeta:=\alpha\mu+\varphi 
\quad \mbox{in}\ L^2(0, T; H) 
\end{align}
as $\beta=\beta_{j}\searrow0$ 
by establishing 
the $L^2(0, T; H^1(\Omega)) \cap H^1(0, T; (H^1(\Omega))^{*})$-estimate 
for $\zeta_{\beta}$  
and by applying the Aubin--Lions lemma 
for the compact embedding 
$H^1(\Omega) \hookrightarrow L^2(\Omega)$ 
and the continuous embedding 
$L^2(\Omega) \hookrightarrow (H^1(\Omega))^{*}$.  
Moreover, they proved that 
\begin{align}\label{previousmethod}
&\int_{\Omega} \bigl((\zeta_{\beta}-\zeta_{\eta})
                         -\alpha(\beta\partial_{t}\varphi_{\beta} 
                                                   + \eta\partial_{t}\varphi_{\eta})\bigr)
                                                                   (\varphi_{\beta} - \varphi_{\eta}) 
\\ \notag 
&= \|\varphi_{\beta} - \varphi_{\eta}\|_{L^2(\Omega)}^2 
    + \alpha\|\nabla (\varphi_{\beta} - \varphi_{\eta})\|_{L^2(\Omega)}^2   
    + \alpha\int_{\Omega} (\xi_{\beta} - \xi_{\eta})(\varphi_{\beta} - \varphi_{\eta}) 
\\ \notag 
    &\,\quad+ \alpha\int_{\Omega} (\pi(\varphi_{\beta}) - \pi(\varphi_{\eta}))
                                                                  (\varphi_{\beta} - \varphi_{\eta}), 
\end{align}  
and hence they could see that 
$\{\varphi_{\beta}\}_{\beta}$ 
satisfies Cauchy's criterion in $L^2(0, T; H)$ 
and then could obtain 
a strong convergence of $\varphi_{\beta}$ to $\varphi$ in $L^2(0, T; H)$ 
which implies that 
$\pi(\varphi_{\beta}) \to \pi(\varphi)$, 
$p(\varphi_{\beta}) \to p(\varphi)$ in $L^2(0, T; H)$ 
as $\beta=\beta_{j}\searrow0$. 

 In the case that 
 $$
 B(r)=\frac{1}{4}\frac{d}{dr}((r^2-1)^{+})^2,\ 
 \widehat{B}(r)=\frac{1}{4}((r^2-1)^{+})^2,\  
 \widehat{\pi}(r)=\frac{1}{4}((1-r^2)^{+})^2,
 $$  
 (J3)-(J5) hold, $G(r)=\frac{1}{4}(r^2-1)^2$, that is, 
 $G$ is the classical double well potential, and 
 $G'(r)=r^3-r$.    
 The $L^{\infty}(0, T; H^1(\Omega))$-estimate for $\varphi_{\beta}$ 
 can be established 
 by using the Poincar\'e--Wirtinger inequality 
 (see e.g., \cite{CGRS-2015, CGRS-2017}). 
 However, in the case that $\Omega \subset \RN$ is an {\it unbounded} domain, 
 the inequality and 
 the Aubin--Lions lemma cannot be applied directly.   
 Thus we cannot show \eqref{keyofpre} 
 in the case of {\it unbounded} domains.    
 Moreover, 
 the classical double well potential does not satisfy (J5) 
 in the case of {\it unbounded} domains.  
 

\subsection{Motivation of this work} \label{Subsec1.2} 

Cahn--Hilliard equations on unbounded domains 
were studied by a few authors (see e.g., \cite{B-2006, FKY-2017, KY1, KY2}).  
In particular, 
Cahn--Hilliard type field systems related to tumor growth on unbounded domains 
have not been studied yet. 
The case of unbounded domains 
has the difficult mathematical point that 
compactness methods cannot be applied directly. 
It would be interesting to construct an applicable theory for the case of 
unbounded domains and to set assumptions for the case of unbounded domains 
by trying to keep some typical examples 
in previous works, that is, in the case of bounded domains   
as much as possible. 
By considering the case of unbounded domains, 
it would be possible to make a new finding 
which is not made in the case of bounded domains.  
Also, the new finding would be useful 
for other studies of partial differential equations. 
This article considers the initial-boundary value problem 
on a {\it bounded} or an {\it unbounded} domain for the limit system    
%
%
 \begin{equation}\label{P}\tag{P}
  \begin{cases} 
    \alpha\partial_{t} \mu + \partial_{t} \varphi 
         -\Delta\mu = p(\sigma - \mu)  
         & \mbox{in}\ \Omega\times(0, T),
 \\[3mm]
         \mu = (-\Delta+1)\varphi + \xi + \pi(\varphi),\ \xi \in B(\varphi)
         & \mbox{in}\ \Omega\times(0, T),
 \\[3mm]
         \partial_{t} \sigma  
         -\Delta\sigma = -p(\sigma - \mu)   
         & \mbox{in}\ \Omega\times(0, T),
 \\[3mm]
         \partial_{\nu}\mu = \partial_{\nu}\varphi = \partial_{\nu}\sigma = 0                                   
         & \mbox{on}\ \partial\Omega\times(0, T),
 \\[2.5mm]
        \mu(0) = \mu_0, \varphi(0) = \varphi_{0}, \sigma(0) = \sigma_0                                          
         & \mbox{in}\ \Omega
  \end{cases}
\end{equation}
by passing to the limit in the following system as $\beta\searrow0$: 
%
%
%
\begin{equation}\label{Pbeta}\tag*{(P)$_{\beta}$}
  \begin{cases} 
    \alpha\partial_{t} \mu_{\beta}  
         + \partial_{t} \varphi_{\beta}  
         -\Delta\mu_{\beta} = p(\sigma_{\beta} - \mu_{\beta})  
         & \mbox{in}\ \Omega\times(0, T),
 \\[3mm]
         \mu_{\beta} = \beta\partial_{t} \varphi_{\beta}  
                            + (-\Delta+1)\varphi_{\beta} + \xi_{\beta} 
                            + \pi(\varphi_{\beta}),\ \xi_{\beta} \in B(\varphi_{\beta})
         & \mbox{in}\ \Omega\times(0, T),
 \\[3mm]
         \partial_{t} \sigma_{\beta} 
         -\Delta\sigma_{\beta} = -p(\sigma_{\beta} - \mu_{\beta})   
         & \mbox{in}\ \Omega\times(0, T),
 \\[3mm]
         \partial_{\nu}\mu_{\beta} = \partial_{\nu}\varphi_{\beta} 
         = \partial_{\nu}\sigma_{\beta} = 0                                   
         & \mbox{on}\ \partial\Omega\times(0, T),
 \\[2.5mm]
        \mu_{\beta}(0) = \mu_0, \varphi_{\beta}(0) = \varphi_{0}, 
        \sigma_{\beta}(0) = \sigma_0                                          
         & \mbox{in}\ \Omega,
  \end{cases}
\end{equation}
where $\Omega$ is a {\it bounded} or an {\it unbounded} domain 
in $\RN$ ($N \in{\mathbb N}$) with smooth bounded boundary $\partial\Omega$ 
(e.g., $\Omega = \mathbb{R}^{N}\setminus \overline{B(0, R)}$,  
where $B(0, R)$ is the open ball with center $0$ and radius $R>0$) or $\Omega=\mathbb{R}^{N}$ or $\Omega=\mathbb{R}_{+}^{N}$, 
$p\geq0$, $\alpha>0$ and $\beta >0$, 
under the following conditions (C1)-(C4):   
%
%
%
 \begin{enumerate} 
 \setlength{\itemsep}{0mm}
 \item[(C1)] $B \subset \mathbb{R}\times\mathbb{R}$                                
 is a maximal monotone graph with effective domain $D(B)$ 
 and $B(r) = \partial\widehat{B}(r)$, where 
 $\partial\widehat{B}$ denotes the subdifferential of 
 a proper lower semicontinuous convex function 
 $\widehat{B} : \mathbb{R} \to [0, +\infty]$ satisfying $\widehat{B}(0) = 0$.    
 \item[(C2)] $\pi : \mathbb{R} \to \mathbb{R}$ is                         
 a Lipschitz continuous function and $\pi(0) = 0$. 
 Moreover, there exists a function $\widehat{\pi} \in C^1(\mathbb{R})$ 
 such that $\pi = \widehat{\pi}'$ and $\widehat{\pi}(0)=0$. 
 \item[(C3)] $\mu_{0}, \sigma_0 \in L^2(\Omega)$, 
 $\varphi_0 \in H^1(\Omega)$  
 and $G(\varphi_{0}) \in L^1(\Omega)$, 
 where $G:=\widehat{B}+\widehat{\pi}$.   
 \item[(C4)] $G(r) + \frac{\|\pi'\|_{L^{\infty}(\mathbb{R})}}{2}r^2 \geq 0$  
 for all $r \in \mathbb{R}$ 
 and $\|\pi'\|_{L^{\infty}(\mathbb{R})}<1$. 
 \end{enumerate} 
In the case that 
\begin{align*}
&G(r)= C_{G}(r^4-2r^2), 
\\ 
&B(r)=4C_{G}r^3,\ \widehat{B}(r) = C_{G}r^4, 
\\ 
&\pi(r)= -4C_{G}r,\ \widehat{\pi}(r)=-2C_{G}r^2, 
\end{align*} 
where $C_{G} \in (0, \frac{1}{4})$ is a constant,  
(C1)-(C4) hold    
and $G'(r)=4C_{G}(r^3-r)$    
(see Section \ref{Subsec1.3}).  

\smallskip

%
%
%
This article puts the Hilbert spaces 
   $$
   H:=L^2(\Omega), \quad V:=H^1(\Omega)
   $$
 with inner products $(u_{1}, u_{2})_{H}:=\int_{\Omega}u_{1}u_{2}\,dx$ 
 ($u_{1}, u_{2} \in H$)  
 and $(v_{1}, v_{2})_{V}:=
 \int_{\Omega}\nabla v_{1}\cdot\nabla v_{2}\,dx + \int_{\Omega} v_{1}v_{2}\,dx$ 
 ($v_{1}, v_{2} \in V$), 
 respectively, 
 and with norms $\|u\|_{H}:=(u, u)_{H}^{1/2}$ ($u\in H$) and 
 $\|v\|_{V}:=(v, v)_{V}^{1/2}$ ($v\in V$), respectively.  
 Moreover, this paper uses  
   $$
   W:=\bigl\{z\in H^2(\Omega)\ |\ \partial_{\nu}z = 0 \quad 
   \mbox{a.e.\ on}\ \partial\Omega\bigr\}.
   $$
 The notation $V^{*}$ denotes the dual space of $V$ with 
 duality pairing $\langle\cdot, \cdot\rangle_{V^*, V}$. 
 Moreover, in this paper, a bijective mapping $F : V \to V^{*}$ and 
 the inner product in $V^{*}$ are defined as 
    \begin{align}
    &\langle Fv_{1}, v_{2} \rangle_{V^*, V} := 
    (v_{1}, v_{2})_{V} \quad \mbox{for all}\ v_{1}, v_{2}\in V, 
    \label{defF}
    \\[1mm]
    &(v_{1}^{*}, v_{2}^{*})_{V^{*}} := 
    \left\langle v_{1}^{*}, F^{-1}v_{2}^{*} 
    \right\rangle_{V^*, V} 
    \quad \mbox{for all}\ v_{1}^{*}, v_{2}^{*}\in V^{*};
    \label{innerVstar}
    \end{align}
 note that $F : V \to V^{*}$ is well-defined by 
 the Riesz representation theorem.  


\subsection{Example} \label{Subsec1.3} 

This article presents the example: 
\begin{align*}
&G(r)= C_{G}(r^4-2r^2), 
\\ 
&B(r)=4C_{G}r^3,\ \widehat{B}(r) = C_{G}r^4, 
\\ 
&\pi(r)= -4C_{G}r,\ \widehat{\pi}(r)=-2C_{G}r^2, 
\end{align*}
where $C_{G}\in(0, \frac{1}{4})$ is a constant.  
These functions satisfy (C1)-(C4). 
Indeed, we have 
 $$
 B(r) = 4C_{G}r^3 = \partial \widehat{B}(r) = \widehat{B}'(r),    
 $$
 which implies (C1).    
 Also, we see that 
 \begin{align*}
 &\pi(r) = -4C_{G}r = \widehat{\pi}'(r), \\[3mm] 
 &|\pi'(r)|=|-4C_{G}|=4C_{G} < 1,       \\[1mm]
 &G(r) + \frac{\|\pi'\|_{L^{\infty}(\mathbb{R})}}{2}r^2 
   = C_{G}r^4 \geq 0,  
 \end{align*}
 and hence (C2) and (C4) hold. 

Therefore (C1)-(C4) hold 
for the functions $G$, $B$, $\widehat{B}$, $\pi$ and $\widehat{\pi}$  
in the example. 


\subsection{Main result for \ref{Pbeta}} \label{Subsec1.4} 

This paper defines weak solutions of \ref{Pbeta} as follows.
%
%
%
%
 \begin{df}        
 A quadruple $(\mu_{\beta}, \varphi_{\beta}, \sigma_{\beta}, \xi_{\beta})$ 
 with 
    \begin{align*}
    &\mu_{\beta}, \sigma_{\beta} \in H^1(0, T; V^{*}) \cap L^2(0, T; V), 
    \\
    &\varphi_{\beta} \in H^1(0, T; H) \cap L^2(0, T; W), 
    \\
    &\xi_{\beta} \in L^2(0, T; H) 
    \end{align*}
 is called a {\it weak solution} of \ref{Pbeta} if 
 $(\mu_{\beta}, \varphi_{\beta}, \sigma_{\beta}, \xi_{\beta})$ 
 satisfies 
    \begin{align}
        & \alpha\bigl\langle (\mu_{\beta})_t, v\bigr\rangle_{V^{*}, V} 
          + \bigl(\partial_{t}\varphi_{\beta}, v\bigr)_{H} 
          + \bigl(\nabla\mu_{\beta}, \nabla v\bigr)_{H} 
          = 
          p(\sigma_{\beta} - \mu_{\beta}, v)_{H} \label{defsolPbe1} \\ \notag 
         &\hspace{9cm}\mbox{a.e.\ on}\ (0, T)\quad  \mbox{for all }\ v\in V,  
     \\[2mm]
        & \mu_{\beta} = \beta\partial_{t}\varphi_{\beta} 
         + (-\Delta + 1)\varphi_{\beta} 
         + \xi_{\beta} + \pi(\varphi_{\beta})\hspace{0.5em}    
         \mbox{and}\hspace{0.5em} \xi_{\beta} \in B(\varphi_{\beta})\quad  
         \mbox{a.e.\ on}\ \Omega\times(0, T), \label{defsolPbe2}
     \\[2mm]
         & \bigl\langle (\sigma_{\beta})_t, v\bigr\rangle_{V^{*}, V} + 
          \bigl(\nabla\sigma_{\beta}, \nabla v\bigr)_{H} 
          = 
          -p(\sigma_{\beta} - \mu_{\beta}, v)_{H}   \quad 
          \mbox{a.e.\ on}\ (0, T)\quad  \mbox{for all }\ v\in V, 
          \label{defsolPbe3}
     \\[2mm]
         & \mu_{\beta}(0) = \mu_0,\ \varphi_{\beta}(0) = \varphi_{0},\ 
        \sigma_{\beta}(0) = \sigma_0   \quad \mbox{a.e.\ on}\ \Omega. 
     \label{defsolPbe4} 
     \end{align}
 \end{df}

This article has two main theorems. 
The first main result is concerned with 
existence and uniqueness of solutions to \ref{Pbeta}. 
 \begin{thm}\label{maintheorem1}
 Assume that  {\rm (C1)-(C4)} hold.  
 Then 
 there exists $\alpha_{0} \in (0, 1)$ such that 
 for all $\alpha \in (0, \alpha_{0})$ and all $\beta \in (0, 1)$    
 there exists a unique weak solution 
 $(\mu_{\beta}, \varphi_{\beta}, \sigma_{\beta}, \xi_{\beta})$ of {\rm \ref{Pbeta}}  
 satisfying                                                 
    \begin{align*}
    &\mu_{\beta}, \sigma_{\beta} \in H^1(0, T; V^{*}) \cap L^2(0, T; V), 
    \\
    &\varphi_{\beta} \in H^1(0, T; H) \cap L^2(0, T; W), 
    \\
    &\xi_{\beta} \in L^2(0, T; H).
    \end{align*}
 Moreover, there exists a constant $M_{1}=M_{1}(T)>0$ such that  
     \begin{align}\label{betaes1}
     &\alpha^{1/2}\|\mu_{\beta}\|_{L^{\infty}(0, T; H)} 
       + \|\nabla \mu_{\beta}\|_{L^2(0, T; H)} 
       + \beta^{1/2}\|\partial_{t}\varphi_{\beta}\|_{L^2(0, T; H)} 
       + \|\varphi_{\beta}\|_{L^{\infty}(0, T; V)} \\ \notag 
     &+ \|(\alpha\mu_{\beta}+\varphi_{\beta})_{t}\|_{L^2(0, T; V^{*})} 
     + \|\sigma_{\beta}\|_{H^1(0, T; V^{*}) \cap L^{\infty}(0, T; H) \cap L^2(0, T; V)} 
     \\ \notag 
     &\leq M_{1}\Bigl(\alpha^{1/2}\|\mu_0\|_{H} + \|\varphi_0\|_{V} 
                     + \|G(\varphi_0)\|_{L^1(\Omega)}^{1/2} 
                     + \|\sigma_0\|_{H} \Bigr) 
     \end{align}
and 
   \begin{align}\label{betaes2}
     &\|\mu_{\beta}\|_{L^2(0, T; V)} + \|\varphi_{\beta}\|_{L^2(0, T; W)} 
     + \|\xi_{\beta}\|_{L^2(0, T; H)}
     \\ \notag 
     &\leq M_{1}\Bigl(\alpha^{1/2}\|\mu_0\|_{H} + \|\varphi_0\|_{V} 
                     + \|G(\varphi_0)\|_{L^1(\Omega)}^{1/2} 
                     + \|\sigma_0\|_{H} + \|\mu_{\beta}\|_{L^2(0, T; H)} \Bigr)
     \end{align}
for all $\alpha \in (0, \alpha_{0})$ and all $\beta \in (0, 1)$. 
 \end{thm}


\subsection{Main results for \eqref{P} and error estimates} \label{Subsec1.5}

This article defines weak solutions of \eqref{P} as follows.
%
%
%
 \begin{df}         
 A quadruple $(\mu, \varphi, \sigma, \xi)$ with 
    \begin{align*}
    &\mu \in L^{\infty}(0, T; H) \cap L^2(0, T; V), 
    \\
    &\varphi \in L^{\infty}(0, T; V) \cap L^2(0, T; W), 
    \\
    &\alpha\mu+\varphi \in H^1(0, T; V^{*}), 
    \\
    &\sigma \in H^1(0, T; V^{*}) \cap L^2(0, T; V), 
    \\
    &\xi \in L^2(0, T; H)
    \end{align*}
 is called a {\it weak solution} of \eqref{P} 
 if $(\mu, \varphi, \sigma, \xi)$ satisfies 
    \begin{align}
        & \bigl\langle (\alpha\mu+\varphi)_t, v\bigr\rangle_{V^{*}, V} + 
          \bigl(\nabla\mu, \nabla v\bigr)_{H} 
          = 
          p(\sigma - \mu, v)_{H}   \quad 
          \mbox{a.e.\ on}\ (0, T)\quad  \mbox{for all }\ v\in V, 
          \label{defsolP1}
     \\[2mm]
        & \mu = (-\Delta + 1)\varphi 
         + \xi + \pi(\varphi)\quad  
         \mbox{and}\quad \xi \in B(\varphi)\qquad  
         \mbox{a.e.\ on}\ \Omega\times(0, T),  \label{defsolP2}
     \\[2mm]
        & \bigl\langle \sigma_t, v\bigr\rangle_{V^{*}, V} + 
          \bigl(\nabla\sigma, \nabla v\bigr)_{H} 
          = 
          -p(\sigma - \mu, v)_{H}   \quad 
          \mbox{a.e.\ on}\ (0, T)\quad  \mbox{for all }\ v\in V, 
          \label{defsolP3}
     \\[2mm]
        & (\alpha\mu+\varphi)(0) = \alpha\mu_{0} + \varphi_0,\ 
           \sigma(0) = \sigma_0 \quad \mbox{a.e.\ on}\ \Omega. 
         \label{defsolP4}
     \end{align}
 \end{df}

The second main result asserts existence and uniqueness of solutions to \eqref{P} 
and the error estimate 
between the solution of \eqref{P} and the solution of \ref{Pbeta}. 
 \begin{thm}\label{maintheorem2}
 Assume {\rm (C1)-(C4)} and let $\alpha_0$ be as in Theorem \ref{maintheorem1}. 
Then there exists $\alpha_{00} \in (0, \alpha_{0}]$ such that 
for all $\alpha \in (0, \alpha_{00})$ 
there exists a unique weak solution $(\mu, \varphi, \sigma, \xi)$ of 
 {\rm \eqref{P}} satisfying
     \begin{align*}
    &\mu \in L^{\infty}(0, T; H) \cap L^2(0, T; V), 
    \\
    &\varphi \in L^{\infty}(0, T; V) \cap L^2(0, T; W), 
    \\
    &\alpha\mu+\varphi \in H^1(0, T; V^{*}), 
    \\
    &\sigma \in H^1(0, T; V^{*}) \cap L^2(0, T; V), 
    \\
    &\xi \in L^2(0, T; H). 
    \end{align*}
Moreover, for all $\alpha \in (0, \alpha_{00})$ there exists a constant $M_{2}=M_{2}(\alpha, T)>0$ such that 
     \begin{align}\label{eres}
     &\|\mu_{\beta}-\mu\|_{L^2(0, T; H)} + \|\varphi_{\beta}-\varphi\|_{L^2(0, T; V)} 
     + \|\sigma_{\beta}-\sigma\|_{L^{\infty}(0, T; H) \cap L^2(0, T; V)} \\ \notag  
     &+\|
           (\alpha\mu_{\beta}+\varphi_{\beta}+\sigma_{\beta}) 
           - (\alpha\mu+\varphi+\sigma)
         \|_{L^{\infty}(0, T; V^{*})}  
     \leq M_{2}\beta^{1/2} 
     \end{align}
 for all $\beta \in (0, 1)$. 
 \end{thm}


\subsection{Outline of this paper} \label{Subsec1.6}  

Though the main theorems of this work are almost the same as 
\cite[Theorems 2.2, 2.3 and 2.7]{CGRS-2017}, 
we cannot prove Theorems \ref{maintheorem1} and \ref{maintheorem2} 
in the same way as in the previous work (\cite{CGRS-2017}) 
because the embedding $H^1(\Omega) \hookrightarrow L^2(\Omega)$ 
is not compact in the case that $\Omega$ is an unbounded domain. 
Therefore this paper constructs an applicable theory 
for not only the case of bounded domains 
but also the case of unbounded domains. 
The strategy in the proof of Theorem \ref{maintheorem1} is as follows. 
To establish existence of solutions to \ref{Pbeta} 
we consider the approximation 
\begin{equation}\label{Pbeplam}\tag*{(P)$_{\beta, \ep, \lambda}$}
  \begin{cases} 
    \alpha\partial_{t} \mu_{\beta, \ep, \lambda} 
         + \partial_{t} \varphi_{\beta, \ep, \lambda}  
         +(-\Delta)_{\lambda}\mu_{\beta, \ep, \lambda} 
         = p(\sigma_{\beta, \ep, \lambda} 
              - \mu_{\beta, \ep, \lambda})  
         & \mbox{in}\ \Omega\times(0, T),
 \\[0mm]
         \mu_{\beta, \ep, \lambda} 
         = \beta\partial_{t}\varphi_{\beta, \ep, \lambda} 
            + ((-\Delta)_{\lambda}+1)\varphi_{\beta, \ep, \lambda} 
            + G_{\ep}'(\varphi_{\beta, \ep, \lambda}) 
         & \mbox{in}\ \Omega\times(0, T),
 \\[1mm]
         \partial_{t} \sigma_{\beta, \ep, \lambda}  
         +(-\Delta)_{\lambda}\sigma_{\beta, \ep, \lambda} 
         = -p(\sigma_{\beta, \ep, \lambda} 
                - \mu_{\beta, \ep, \lambda})   
         & \mbox{in}\ \Omega\times(0, T), 
 \\[0mm]
        \mu_{\beta, \ep, \lambda}(0) = \mu_{0\ep}, 
        \varphi_{\beta, \ep, \lambda}(0) = \varphi_{0}, 
        \sigma_{\beta, \ep, \lambda}(0) = \sigma_{0\ep}                                          
         & \mbox{in}\ \Omega,
  \end{cases}
\end{equation}
where $\ep>0$, $\lambda > 0$, $(-\Delta)_{\lambda}$ is the Yosida approximation 
of the Neumann Laplacian $-\Delta$,  
$G_{\ep} = \widehat{B_{\ep}} + \widehat{\pi}$,    
$\widehat{B_{\ep}} : \mathbb{R} \to \mathbb{R}$ is 
the Moreau--Yosida regularization of $\widehat{B}$ 
(see Remark \ref{MYreg}), 
$\mu_{0\ep}:=(1-\ep\Delta)^{-1/2}\mu_{0}$ and 
$\sigma_{0\ep}:=(1-\ep\Delta)^{-1/2}\sigma_{0}$. 
We can show that 
there exists a unique solution $(\mu_{\beta, \ep, \lambda}, 
           \varphi_{\beta, \ep, \lambda}, \sigma_{\beta, \ep, \lambda})$ 
of \ref{Pbeplam} such that 
$\mu_{\beta, \ep, \lambda}, 
\varphi_{\beta, \ep, \lambda}, 
\sigma_{\beta, \ep, \lambda} \in C^{1}([0, T]; H)$ 
by applying the Cauchy--Lipschitz--Picard theorem. 
The key to the proof of existence of solutions to \ref{Pbep} 
(see Definition \ref{defsolPbep}) is to prove that 
$\{\varphi_{\beta, \ep, \lambda}\}_{\lambda}$, 
$\{\partial_{t}\varphi_{\beta, \ep, \lambda}\}_{\lambda}$,    
$\{(-\Delta)_{\lambda}\varphi_{\beta, \ep, \lambda}\}_{\lambda}$ 
are bounded in $L^2(0, T; H)$ 
and 
\begin{align*}
\varphi_{\beta, \ep, \lambda} 
= \lambda(-\Delta)_{\lambda}\varphi_{\beta, \ep, \lambda} 
   + (1-\lambda\Delta)^{-1}\varphi_{\beta, \ep, \lambda} 
\to \varphi_{\beta, \ep} 
\quad \mbox{in}\ L^2(0, T; L^2(D))   
\end{align*} 
as $\lambda=\lambda_{j}\searrow0$  
by using the Aubin--Lions lemma for 
$\{(1-\lambda\Delta)^{-1}\varphi_{\beta, \ep, \lambda}\}_{\lambda}$ 
and the compact embedding $H^1(D) \hookrightarrow L^2(D)$, 
where $D \subset \Omega$ is a bounded domain with smooth boundary.    
In particular, the key to showing the initial condition in \ref{Pbep}  
is to use the operator $(1-\Delta)^{-1/2} : H \to V$ and 
the compact embedding 
$H^1(E) \hookrightarrow L^2(E)$, 
where $E \subset \Omega$ is an arbitrary bounded domain with 
smooth boundary.   
Indeed, we can obtain that 
\begin{align*}
(1-\Delta)^{-1/2}\mu_{\beta, \ep, \lambda} 
\to (1-\Delta)^{-1/2}\mu_{\beta, \ep} 
\quad &\mbox{in}\ C([0, T]; L^2(E)), 
\\ 
(1-\Delta)^{-1/2}\varphi_{\beta, \ep, \lambda} 
\to (1-\Delta)^{-1/2}\varphi_{\beta, \ep} 
\quad &\mbox{in}\ C([0, T]; L^2(E)), 
\\   
(1-\Delta)^{-1/2}\sigma_{\beta, \ep, \lambda} 
\to (1-\Delta)^{-1/2}\sigma_{\beta, \ep} 
\quad &\mbox{in}\ C([0, T]; L^2(E)) 
\end{align*}
as $\lambda=\lambda_{j}\searrow0$ 
by applying the Ascoli--Arzela theorem for the compact embedding 
$H^1(E) \hookrightarrow L^2(E)$, 
and hence we can verify the initial condition in \ref{Pbep}.  
At the moment, we do not know whether 
the strong convergence 
\begin{align}\label{KeyforTheorem1.1}
\varphi_{\beta, \ep} \to \varphi_{\beta}  
\quad \mbox{in}\ L^2(0, T; L^2(D))   
\end{align} 
as $\ep=\ep_{j}\searrow0$,  
where $D \subset \Omega$ is a bounded domain with 
smooth boundary,    
can be proved as in  
\cite[Proof of Theorem 2.2]{CGRS-2017} or not. 
Indeed, by calculating 
$\int_{D} (\zeta_{\beta, \ep}-\zeta_{\beta, \ep'})
                                                 (\varphi_{\beta, \ep} - \varphi_{\beta, \ep'})$ 
as in \eqref{previousmethod},  
where $\zeta_{\beta, \ep}:=\alpha\mu_{\beta, \ep} + \varphi_{\beta, \ep}$, 
the term  
$$
- \alpha\int_{\partial D} 
           (\varphi_{\beta, \ep} - \varphi_{\beta, \ep'})
             \nabla(\varphi_{\beta, \ep} - \varphi_{\beta, \ep'})\cdot \nu_{\partial D}
$$ 
appears because of integration by parts on $D$.  
However, it would be difficult to estimate this term properly.    
Therefore, in this paper, noting that 
$$
\varphi_{\beta, \ep} = \ep(-\Delta)_{\ep}\varphi_{\beta, \ep} 
                 + (1-\ep\Delta)^{-1}\varphi_{\beta, \ep},  
$$
we obtain \eqref{KeyforTheorem1.1} by proving that 
$\{\varphi_{\beta, \ep}\}_{\ep}$, 
$\{\partial_{t}\varphi_{\beta, \ep}\}_{\ep}$, 
$\{-\Delta\varphi_{\beta, \ep}\}_{\ep}$ 
are bounded in $L^2(0, T; H)$ 
and by using the Aubin--Lions lemma for 
$\{(1-\ep\Delta)^{-1}\varphi_{\beta, \ep}\}_{\ep}$  
and the compact embedding $H^1(D) \hookrightarrow L^2(D)$.  
To confirm that $\mu_{\beta}(0)=\mu_{0}$ and $\sigma_{\beta}(0)=\sigma_{0}$ 
in $H$ 
we use not the operator $(1-\Delta)^{-1/2} : H \to V$ 
but the operator $\tilde{J}^{1/2}_{1} : V^{*} \to H$ 
($\tilde{J}_{\lambda}:=\bigl(I + \lambda\tilde{A}\bigr)^{-1}$, $\tilde{A}:=F-I$)  
because $\{\partial_{t}\mu_{\beta, \ep}\}_{\ep}$,  
$\{\partial_{t}\sigma_{\beta, \ep}\}_{\ep}$  
are bounded not in $L^2(0, T; H)$ but in $L^2(0, T; V^{*})$.   
The strategy in the proof of Theorem \ref{maintheorem2} is as follows. 
The key to the proof of existence of solutions to \eqref{P} 
is to obtain a strong convergence of $\varphi_{\beta}$. 
Indeed, we can confirm Cauchy's criterion for solutions of \ref{Pbeta}  
in reference to \cite[Proof of Theorem 2.3]{CGRS-2017}.   
The key to verifying \eqref{defsolP4}  
is to use the operator $\tilde{J}^{1/2}_{1} : V^{*} \to H$ and 
the compact embedding 
$H^1(E) \hookrightarrow L^2(E)$,  
where $E \subset \Omega$ is an arbitrary bounded domain with 
smooth boundary.   
Indeed, since $\{(\alpha\mu_{\beta}+\varphi_{\beta})_{t}\}_{\beta}$, 
$\{(\sigma_{\beta})_{t}\}_{\beta}$ are bounded in $L^2(0, T; V^{*})$, 
we can obtain that 
\begin{align*}
\tilde{J}^{1/2}_{1}(\alpha\mu_{\beta}+\varphi_{\beta})
\to \tilde{J}^{1/2}_{1}(\alpha\mu+\varphi),\  
\tilde{J}^{1/2}_{1}\sigma_{\beta} \to \tilde{J}^{1/2}_{1}\sigma
\quad \mbox{in}\ C([0, T]; L^2(E)) 
\end{align*}  
as $\beta=\beta_{j}\searrow0$ 
by applying the Ascoli--Arzela theorem for the compact embedding 
$H^1(E) \hookrightarrow L^2(E)$.  
Thus we can show \eqref{defsolP4}.   

This paper is organized as follows. 
In Section \ref{Sec2} we give useful results for 
proving the main theorems. 
Sections \ref{Sec3} and \ref{Sec4} are devoted to the proofs of 
Theorems \ref{maintheorem1} and \ref{maintheorem2}, respectively.

 \section{Preliminaries}\label{Sec2}

In this section we will provide some results which will be used later for 
the proofs of Theorems \ref{maintheorem1} and \ref{maintheorem2}. 
 \begin{lem}[{\cite[Section 8, Corollary 4]{Si-1987}}] \label{Ascoli}                                                                   
 Assume that 
 $$
 X \subset Z \subset Y \ \mbox{with compact embedding}\ X \hookrightarrow Z\ 
 \mbox{$($$X$, $Z$ and $Y$ are Banach spaces$)$.}
 $$
 \begin{enumerate}
 \setlength{\itemsep}{-0.5mm}
 \item[{\rm (i)}] Let $K$ be bounded in $L^p(0, T; X)$ and 
 $\{\frac{\partial v}{\partial t}\ |\ v\in K \}$ 
 be bounded in $L^1(0, T; Y)$ with some constant $1\leq p<\infty$. 
 Then $K$ is relatively compact in $L^p(0, T; Z)$. 
 \item[{\rm (ii)}] Let $K$ be bounded in $L^{\infty}(0, T; X)$ and 
 $\{\frac{\partial v}{\partial t}\ |\ v\in K \}$ 
 be bounded in $L^r(0, T; Y)$ with some constant $r>1$. 
 Then $K$ is relatively compact in $C([0, T]; Z)$. 
 \end{enumerate}
 \end{lem}
%
 \begin{lem}\label{ra}
  Let $\lambda>0$ and put 
 \begin{align*}
 &J_{\lambda}:=(I - \lambda \Delta)^{-1} : H \to H, \quad
    (-\Delta)_{\lambda}:= \frac{1}{\lambda}(I - J_{\lambda}) : H \to H,  
 \\ 
 &\tilde{A}:=F-I : V \to V^{*}, \quad 
 \tilde{J}_{\lambda}:=
 \bigl(I + \lambda\tilde{A}\bigr)^{-1} : V^{*} \to V^{*}. 
 \end{align*}
 Then we have 
 \begin{align}
 &\tilde{J}_{\lambda}|_{H} = J_{\lambda}, \label{tool1}\\
 &\|J^{1/2}_{\lambda}v\|_{H} \leq \|v\|_{H}, \label{tool2}\\
 &\|J^{1/2}_1v\|_{V} = \|v\|_{H}, \label{tool3} \\
 &\|\tilde{J}^{1/2}_1v^*\|_{H} = \|v^*\|_{V^*} \label{toolplus}
 \end{align}
 for all $v \in H$ and all $v^{*} \in V^{*}$, 
 and 
 \begin{align}\label{tool4}
 \|(-\Delta)^{1/2}_{\lambda} v\|_{H} \leq \|v\|_{V} 
 \end{align}
 for all $v \in V$, 
 where 
 $-\Delta : W \subset H \to H$ is the Neumann Laplacian. 
 \end{lem}
\begin{proof}
We can show \eqref{tool1} by the same argument as in \cite[Lemma 3.3]{KY2}. 
Also, we can verify \eqref{tool2}, \eqref{tool3} and \eqref{tool4} 
by the same argument as in \cite[Lemma 3.2]{K(nonlocal)}. 
Now we confirm \eqref{toolplus}. 
Noting that 
$\tilde{J}_{1}=(I+\tilde{A})^{-1}=F^{-1}$ and 
$(\tilde{J}^{1/2}_1 v^{*}, v)_{H} = \langle v^{*}, J^{1/2}_1 v \rangle_{V^{*}, V}$ 
for all $v^{*} \in V^{*}$ and all $v \in H$ (see e.g., \cite[Lemma 3.3]{OSY}), 
we see from \eqref{innerVstar} and \eqref{tool1} that 
$$
\|\tilde{J}^{1/2}_{1}v^*\|^2_{H} 
= (\tilde{J}^{1/2}_{1}v^*, \tilde{J}^{1/2}_{1}v^*)_{H} 
= \langle v^*, J_{1}^{1/2}\tilde{J}_1^{1/2}v^{*} \rangle_{V^*, V}
= \langle v^*, F^{-1}v^{*} \rangle_{V^*, V}     
= \|v^*\|^2_{V^*}                                  
$$
for all $v^* \in V^*$, that is, we can obtain \eqref{toolplus}.
\end{proof} 
%
\begin{lem}[{\cite[Lemma 3.3]{K(nonlocal)}}]\label{keylemma}
Let $E \subset \Omega$ be 
a bounded domain with smooth boundary 
and 
let $\{v_{\lambda}\}_{\lambda} \subset H^1(0, T; H) \cap L^{\infty}(0, T; H)$ 
satisfy that 
$\{v_{\lambda}\}_{\lambda}$ is bounded in $L^{\infty}(0, T; H)$ and 
$\{\partial_{t}v_{\lambda}\}_{\lambda}$ is bounded in $L^2(0, T; H)$. 
Then 
\begin{align*}
&v_{\lambda} \to v\ \mbox{weakly$^{*}$ in}\ L^{\infty}(0, T; H), \\
&J^{1/2}_{1}v_{\lambda} \to J^{1/2}_{1} v\ 
\mbox{in}\ C([0, T]; L^2(E))
\end{align*}
as $\lambda = \lambda_j \searrow0$ 
with some function $v \in L^{\infty}(0, T; H)$. 
\end{lem}
%
%
\begin{lem}[{\cite[Lemma 3.4]{K(nonlocal)}}]\label{keylemma2}
Let $E \subset \Omega$ be 
a bounded domain with smooth boundary 
and 
let $\{v_{\lambda}\}_{\lambda} \subset H^1(0, T; H)$ 
satisfy that 
$\{v_{\lambda}\}_{\lambda}$ is bounded in $L^2(0, T; H)$ and 
$\{\partial_{t}v_{\lambda}\}_{\lambda}$, 
$\{(-\Delta)_{\lambda}v_{\lambda}\}_{\lambda}$ 
are bounded in $L^2(0, T; H)$. 
Then 
\begin{align*}
&v_{\lambda} \to v\ 
\mbox{in}\ L^2(0, T; L^2(E))
\end{align*}
as $\lambda = \lambda_j \searrow0$ 
with some function $v \in L^2(0, T; W)$.
\end{lem}
%
%
\begin{lem}\label{keylemmaVstar}
Let $E \subset \Omega$ be 
a bounded domain with smooth boundary 
and 
let $\{v_{\lambda}\}_{\lambda} \subset H^1(0, T; V^*) \cap L^{\infty}(0, T; H)$ 
satisfy that 
$\{v_{\lambda}\}_{\lambda}$ is bounded in $L^{\infty}(0, T; H)$ and 
$\{(v_{\lambda})_{t}\}_{\lambda}$ is bounded in $L^2(0, T; V^*)$. 
Then 
\begin{align*}
&v_{\lambda} \to v\ \mbox{weakly$^{*}$ in}\ L^{\infty}(0, T; H), \\
&\tilde{J}^{1/2}_{1}v_{\lambda} \to \tilde{J}^{1/2}_{1} v\ 
\mbox{in}\ C([0, T]; L^2(E))
\end{align*}
as $\lambda = \lambda_j \searrow0$ 
with some function $v \in L^{\infty}(0, T; H)$. 
\end{lem}
\begin{proof}
There exists $v \in  L^{\infty}(0, T; H)$ such that 
\begin{align*}
&v_{\lambda} \to v\ 
\mbox{weakly$^{*}$ in}\ L^{\infty}(0, T; H)
\end{align*}
as $\lambda = \lambda_j \searrow0$. 
We see that 
\begin{equation}\label{embedding}
H^1(E) \subset L^2(E) 
\subset L^2(E)\ 
\mbox{with compact embedding}\ 
H^1(E) \hookrightarrow L^2(E).  
\end{equation}
It follows from \eqref{tool1} and \eqref{tool3} that 
$$
\|\tilde{J}^{1/2}_1 v_{\lambda}(t)\|_{H^1(E)} 
\leq \|\tilde{J}^{1/2}_1v_{\lambda}(t)\|_{V} 
= \|J^{1/2}_1v_{\lambda}(t)\|_{V} 
= \|v_{\lambda}(t)\|_{H}.  
$$
Thus  
there exists a constant $C_1>0$ such that 
\begin{equation}\label{Linfty}
\|\tilde{J}^{1/2}_1 v_{\lambda} \|_{L^{\infty}(0, T; H^1(E))} \leq C_1. 
\end{equation}
Also, from \eqref{toolplus} we have that 
$$
\|\tilde{J}^{1/2}_1(v_{\lambda})_{t}(t)\|_{L^2(E)}
\leq \|\tilde{J}^{1/2}_1(v_{\lambda})_{t}(t)\|_{H}
= \|(v_{\lambda})_{t}(t)\|_{V^{*}},
$$
and hence there exists a constant $C_2>0$ such that 
\begin{equation}\label{L2}
\|\tilde{J}^{1/2}_1(v_{\lambda})_{t}\|_{L^2(0, T; L^2(E))} 
\leq C_2.
\end{equation}
Therefore 
applying \eqref{embedding}-\eqref{L2} and Lemma \ref{Ascoli} 
yields that 
\begin{equation}\label{AfterAscoli}
\tilde{J}^{1/2}_1 v_{\lambda} \to w 
\quad \mbox{in}\ C([0, T]; L^2(E))
\end{equation}
as $\lambda = \lambda_j \searrow0$ 
with some function $w \in C([0, T]; L^2(E))$. 
Now we let 
$\psi \in C_{\mathrm c}^{\infty}([0, T] \times \overline{E})$ 
and will show that 
\begin{equation}\label{henbunhou}
\int_0^T \Bigl(\int_{E} \bigl(\tilde{J}^{1/2}_1v(t) -w(t) \bigr)\psi(t) 
                                                                                                 \Bigr)\,dt 
= 0.
\end{equation} 
We see from \eqref{tool1} that 
\begin{align}\label{LL}
\int_0^T \Bigl(\int_{E} 
                                 \bigl(\tilde{J}^{1/2}_1v_{\lambda}(t) \bigr)\psi(t) \Bigr)\,dt
= \int_0^T (J^{1/2}_1v_{\lambda}(t), \psi(t))_{H}\,dt
= \int_0^T (v_{\lambda}(t), J^{1/2}_1\psi(t))_{H}\,dt. 
\end{align}
Here,   
since 
$\psi \in C_{\mathrm c}^{\infty}([0, T]\times \overline{E}) 
\subset C_{\mathrm c}^{\infty}([0, T]\times \overline{\Omega}) 
\subset L^1(0, T; H)$,   
we infer from \eqref{tool2} that 
$$
J^{1/2}_1\psi \in L^1(0, T; H).
$$
Hence it follows from \eqref{tool1} that 
\begin{align}\label{kyokugen}
\int_0^T (v_{\lambda}(t), J^{1/2}_1\psi(t))_{H}\,dt 
\to \int_0^T (v(t), J^{1/2}_1\psi(t))_{H}\,dt
&= \int_0^T (J^{1/2}_1v(t), \psi(t))_{H}\,dt \\ \notag
&= \int_0^T (\tilde{J}^{1/2}_1v(t), \psi(t))_{H}\,dt
\end{align}
as $\lambda = \lambda_j \searrow0$. 
Thus combination of \eqref{AfterAscoli}, \eqref{LL} and \eqref{kyokugen} 
leads to \eqref{henbunhou}. 
Thus we can obtain 
\begin{equation}\label{Afterhenbunhou}
w = \tilde{J}^{1/2}_1v \quad \mbox{a.e.\ in}\ (0, T) \times E.
\end{equation} 
Therefore we 
derive from \eqref{AfterAscoli} and \eqref{Afterhenbunhou} that 
\begin{align*}
\tilde{J}^{1/2}_1v_{\lambda} \to \tilde{J}^{1/2}_1 v 
\quad \mbox{in}\ C([0, T]; L^2(E))
\end{align*}
as $\lambda = \lambda_j \searrow0$. 
\end{proof}


\section{Proof of Theorem \ref{maintheorem1}}\label{Sec3}

To show existence of solutions to \ref{Pbeta} we consider the approximation 
\begin{equation}\label{Pbeplam}\tag*{(P)$_{\beta, \ep, \lambda}$}
  \begin{cases} 
    \alpha\partial_{t} \mu_{\beta, \ep, \lambda} 
         + \partial_{t} \varphi_{\beta, \ep, \lambda}  
         +(-\Delta)_{\lambda}\mu_{\beta, \ep, \lambda} 
         = p(\sigma_{\beta, \ep, \lambda} 
              - \mu_{\beta, \ep, \lambda})  
         & \mbox{in}\ \Omega\times(0, T),
 \\[0mm]
         \mu_{\beta, \ep, \lambda} 
         = \beta\partial_{t}\varphi_{\beta, \ep, \lambda} 
            + ((-\Delta)_{\lambda}+1)\varphi_{\beta, \ep, \lambda} 
            + G_{\ep}'(\varphi_{\beta, \ep, \lambda}) 
         & \mbox{in}\ \Omega\times(0, T),
 \\[1mm]
         \partial_{t} \sigma_{\beta, \ep, \lambda}  
         +(-\Delta)_{\lambda}\sigma_{\beta, \ep, \lambda} 
         = -p(\sigma_{\beta, \ep, \lambda} 
                - \mu_{\beta, \ep, \lambda})   
         & \mbox{in}\ \Omega\times(0, T), 
 \\[0mm]
        \mu_{\beta, \ep, \lambda}(0) = \mu_{0\ep}, 
        \varphi_{\beta, \ep, \lambda}(0) = \varphi_{0}, 
        \sigma_{\beta, \ep, \lambda}(0) = \sigma_{0\ep}                                          
         & \mbox{in}\ \Omega,
  \end{cases}
\end{equation}
where $\ep>0$, $\lambda > 0$, 
$(-\Delta)_{\lambda}$ is the Yosida approximation 
of the Neumann Laplacian $-\Delta$,  
$G_{\ep} = \widehat{B_{\ep}} + \widehat{\pi}$,    
$\widehat{B_{\ep}} : \mathbb{R} \to \mathbb{R}$ is 
the Moreau--Yosida regularization  
of $\widehat{B}$,  
$\mu_{0\ep}:=(1-\ep\Delta)^{-1/2}\mu_{0}$ and 
$\sigma_{0\ep}:=(1-\ep\Delta)^{-1/2}\sigma_{0}$. 
\begin{remark}\label{MYreg} 
The function $\widehat{B_{\ep}} : \mathbb{R} \to \mathbb{R}$ defined by 
$$
\widehat{B_{\ep}}(r):=
\displaystyle\inf_{s\in\mathbb{R}}\left\{\frac{1}{2\ep}|r-s|^2+\widehat{B}(s) \right\} 
\quad \mbox{for}\ r\in \mathbb{R} 
$$ 
is called the Moreau--Yosida regularization  
of $\widehat{B}$.  
It holds that 
$$
\widehat{B_{\ep}}(r)=\frac{1}{2\ep}|r-J_{\ep}^{B}(r)|^2 
+ \widehat{B}(J_{\ep}^{B}(r))
$$ 
for all $r\in\mathbb{R}$, 
where $J_{\ep}^{B}$ is the resolvent operator of $B$  
on $\mathbb{R}$.  
The derivative of $\widehat{B_{\ep}}$ is $B_{\ep}$, 
where $B_{\ep}$ is the Yosida approximation operator of $B$ 
on $\mathbb{R}$, 
and hence  
the identity 
$$
G_{\ep}'(r) = B_{\ep}(r) + \pi(r) 
$$
holds. 
Moreover, 
the inequalities 
$$
0\leq \widehat{B_{\ep}}(r) \leq \widehat{B}(r)
$$ 
hold for all $r\in\mathbb{R}$ (see e.g., \cite[Theorem 2.9, p.\ 48]{Barbu2}). 
\end{remark}
We can prove existence for \ref{Pbeplam}. 
\begin{lem}\label{solPbeplam}
There exists a unique solution $(\mu_{\beta, \ep, \lambda}, 
           \varphi_{\beta, \ep, \lambda}, \sigma_{\beta, \ep, \lambda})$ 
of {\rm \ref{Pbeplam}} such that 
$\mu_{\beta, \ep, \lambda}, 
\varphi_{\beta, \ep, \lambda}, 
\sigma_{\beta, \ep, \lambda} \in C^{1}([0, T]; H)$. 
\end{lem}
\begin{proof}
It is possible to rewrite \ref{Pbeplam} as 
\begin{equation}\label{Q}\tag{Q}
  \begin{cases} 
    \frac{dU}{dt} = L(U)
         & \mbox{on}\ [0, T], 
 \\[1mm]
        U(0)=U_{0}, 
  \end{cases}
\end{equation}
where 
$$
U=\left(
    \begin{array}{c}
      \mu_{\beta, \ep, \lambda} \\
      \varphi_{\beta, \ep, \lambda} \\
      \sigma_{\beta, \ep, \lambda}
    \end{array}
  \right), \
U_{0}=\left(
    \begin{array}{c}
      \mu_{0\ep} \\
      \varphi_{0} \\
      \sigma_{0\ep}
    \end{array}
  \right) 
\in H\times H\times H, 
$$ 
and 
the operator $L : H\times H \times H \to H\times H \times H$ 
is defined as   
\begin{align*}
&L: 
\left(
    \begin{array}{c}
      \mu \\
      \varphi \\
      \sigma
    \end{array}
  \right) 
\mapsto 
\left(
    \begin{array}{c}
      -\frac{1}{\alpha}(-\Delta)_{\lambda}\mu + \frac{p}{\alpha}(\sigma - \mu) 
      - \frac{1}{\alpha\beta}\mu + \frac{1}{\alpha\beta}(-\Delta)_{\lambda}\varphi 
      + \frac{1}{\alpha\beta}\varphi + \frac{1}{\alpha\beta}G_{\ep}'(\varphi) 
      \\[3mm]
      \frac{1}{\beta}\mu - \frac{1}{\beta}(-\Delta)_{\lambda}\varphi 
      - \frac{1}{\beta}\varphi - \frac{1}{\beta}G_{\ep}'(\varphi) 
      \\[3mm] 
      -(-\Delta)_{\lambda}\sigma - p(\sigma - \mu)
    \end{array}
  \right).  
\end{align*} 
Here, noting from Remark \ref{MYreg} that  
\begin{align*}
&\|(-\Delta)_{\lambda}(\phi - \psi)\|_{H} 
\leq \frac{1}{\lambda}\|\phi-\psi\|_{H}, 
\\[1mm]
&\|G_{\ep}'(\phi)-G_{\ep}'(\psi)\|_{H} 
\leq \left(\frac{1}{\ep}+ \|\pi'\|_{L^{\infty}(\mathbb{R})}\right)\|\phi-\psi\|_{H}   
\end{align*} 
for all $\phi \in H$ and all $\psi \in H$,     
we can observe that the operator 
$L : H\times H \times H \to H\times H \times H$ 
is Lipschitz continuous.     
Thus, applying the Cauchy--Lipschitz--Picard theorem,  
we can show that there exists a unique solution   
$U=\left(
    \begin{array}{c}
      \mu_{\beta, \ep, \lambda} \\
      \varphi_{\beta, \ep, \lambda} \\
      \sigma_{\beta, \ep, \lambda}
    \end{array}
  \right) 
\in C^1([0, T]; H\times H\times H)$ 
of \eqref{Q}. 
Therefore we can obtain this lemma. 
\end{proof}
%
\begin{lem}\label{ineqG}
We have 
\begin{align*}
G_{\ep}(r) \geq -\frac{\|\pi'\|_{L^{\infty}(\mathbb{R})}}{2}r^2 
- 2\|\pi'\|_{L^{\infty}(\mathbb{R})}\ep r^2
\end{align*} 
for all $r \in \mathbb{R}$ and all $\ep>0$. 
\end{lem}
\begin{proof}
See \cite[(4.8) in the proof of Lemma 4.2]{K(nonlocal)}.  
\end{proof}
%
\begin{lem}\label{esPbeplam}
There exists $\ep_{1} \in (0, 1)$ such that 
for all $\alpha>0$, $\beta>0$ and $\ep \in (0, \ep_{1})$    
there exist constants $C=C(\alpha, \beta)>0$ 
and $C'=C'(\alpha, \beta, \ep)>0$  satisfying 
\begin{align}
&\|\mu_{\beta, \ep, \lambda}(t)\|_{H}^2 
  + \|\varphi_{\beta, \ep, \lambda}(t)\|_{H}^2 
  + \|\sigma_{\beta, \ep, \lambda}(t)\|_{H}^2 
  + \int_{0}^{t}\|\partial_{t}\varphi_{\beta, \ep, \lambda}(s)\|_{H}^2\,ds 
\leq C, \label{esPbeplam1} 
\\[0mm] 
&\int_{0}^{t}\|\partial_{t}\sigma_{\beta, \ep, \lambda}(s)\|_{H}^2\,ds 
\leq C', \label{esPbeplam2} 
\\[1.5mm]
&\int_{0}^{t}\|\partial_{t}\mu_{\beta, \ep, \lambda}(s)\|_{H}^2\,ds 
\leq C', \label{esPbeplam3} 
\\[1.5mm]
&\int_{0}^{t}\big(\|(-\Delta)_{\lambda}\mu_{\beta, \ep, \lambda}(s)\|_{H}^2  
+ \|(-\Delta)_{\lambda}\varphi_{\beta, \ep, \lambda}(s)\|_{H}^2 
+ \|(-\Delta)_{\lambda}\sigma_{\beta, \ep, \lambda}(s)\|_{H}^2 \bigr)\,ds  
\leq C' \label{esPbeplam4} 
\end{align}
for all $t \in [0, T]$ and all $\lambda>0$. 
\end{lem}
\begin{proof}
Let $\alpha, \beta>0$.    
Then we derive from the first equation in \ref{Pbeplam} that 
\begin{align}\label{popo1}
&\frac{\alpha}{2}\frac{d}{dt}\|\mu_{\beta, \ep, \lambda}(t)\|_{H}^2 
+ (\partial_{t}\varphi_{\beta, \ep, \lambda}(t), 
                                          \mu_{\beta, \ep, \lambda}(t))_{H} 
+ ((-\Delta)_{\lambda}\mu_{\beta, \ep, \lambda}(t), 
                                          \mu_{\beta, \ep, \lambda}(t))_{H} 
\\ \notag
&= p(\sigma_{\beta, \ep, \lambda}(t)-\mu_{\beta, \ep, \lambda}(t), 
                                                          \mu_{\beta, \ep, \lambda}(t))_{H}.  
\end{align}
Here it follows from the second equation in \ref{Pbeplam} that 
\begin{align}\label{popo2}
&(\partial_{t}\varphi_{\beta, \ep, \lambda}(t), 
                                          \mu_{\beta, \ep, \lambda}(t))_{H} 
\\ \notag 
&= \beta\|\partial_{t}\varphi_{\beta, \ep, \lambda}(t)\|_{H}^2 
   + \frac{1}{2}\frac{d}{dt}
           \|(-\Delta)_{\lambda}^{1/2}\varphi_{\beta, \ep, \lambda}(t)\|_{H}^2 
\\ \notag 
&\,\quad+\frac{1}{2}\frac{d}{dt}\|\varphi_{\beta, \ep, \lambda}(t)\|_{H}^2 
  +\frac{d}{dt}\int_{\Omega}G_{\ep}(\varphi_{\beta, \ep, \lambda}(t)).   
\end{align} 
Hence from \eqref{popo1} and \eqref{popo2} we have 
\begin{align}\label{popo3}
&\frac{\alpha}{2}\frac{d}{dt}\|\mu_{\beta, \ep, \lambda}(t)\|_{H}^2 
+ \beta\|\partial_{t}\varphi_{\beta, \ep, \lambda}(t)\|_{H}^2 
   + \frac{1}{2}\frac{d}{dt}
           \|(-\Delta)_{\lambda}^{1/2}\varphi_{\beta, \ep, \lambda}(t)\|_{H}^2 
\\ \notag 
&+\frac{1}{2}\frac{d}{dt}\|\varphi_{\beta, \ep, \lambda}(t)\|_{H}^2 
  +\frac{d}{dt}\int_{\Omega}G_{\ep}(\varphi_{\beta, \ep, \lambda}(t))
+ ((-\Delta)_{\lambda}\mu_{\beta, \ep, \lambda}(t), 
                                          \mu_{\beta, \ep, \lambda}(t))_{H} 
\\ \notag
&= p(\sigma_{\beta, \ep, \lambda}(t)-\mu_{\beta, \ep, \lambda}(t), 
                                                          \mu_{\beta, \ep, \lambda}(t))_{H}. 
\end{align}
On the other hand, we see from the third equation in \ref{Pbeplam} that 
\begin{align}\label{popo4}
&\frac{1}{2}\frac{d}{dt}\|\sigma_{\beta, \ep, \lambda}(t)\|_{H}^2 
+ ((-\Delta)_{\lambda}\sigma_{\beta, \ep, \lambda}(t), 
                                          \sigma_{\beta, \ep, \lambda}(t))_{H} 
\\ \notag
&= -p(\sigma_{\beta, \ep, \lambda}(t)-\mu_{\beta, \ep, \lambda}(t), 
                                                             \sigma_{\beta, \ep, \lambda}(t))_{H}. 
\end{align}
Thus we infer from \eqref{popo3} and \eqref{popo4} that 
\begin{align}\label{popo5}
&\frac{\alpha}{2}\frac{d}{dt}\|\mu_{\beta, \ep, \lambda}(t)\|_{H}^2 
+ \beta\|\partial_{t}\varphi_{\beta, \ep, \lambda}(t)\|_{H}^2 
   + \frac{1}{2}\frac{d}{dt}
           \|(-\Delta)_{\lambda}^{1/2}\varphi_{\beta, \ep, \lambda}(t)\|_{H}^2 
\\ \notag 
&+\frac{1}{2}\frac{d}{dt}\|\varphi_{\beta, \ep, \lambda}(t)\|_{H}^2 
  +\frac{d}{dt}\int_{\Omega}G_{\ep}(\varphi_{\beta, \ep, \lambda}(t))
+ ((-\Delta)_{\lambda}\mu_{\beta, \ep, \lambda}(t), 
                                          \mu_{\beta, \ep, \lambda}(t))_{H} 
\\ \notag
&+\frac{1}{2}\frac{d}{dt}\|\sigma_{\beta, \ep, \lambda}(t)\|_{H}^2 
+ ((-\Delta)_{\lambda}\sigma_{\beta, \ep, \lambda}(t), 
                                          \sigma_{\beta, \ep, \lambda}(t))_{H} 
\\ \notag
&= -p\|\sigma_{\beta, \ep, \lambda}(t)-\mu_{\beta, \ep, \lambda}(t)\|_{H}^2.
\end{align}
Here we have from (C4) and Lemma \ref{ineqG} that 
$1-\|\pi'\|_{L^{\infty}(\mathbb{R})} > 0$ and  
there exists $\ep_{1} \in (0, 1)$ such that 
\begin{align}\label{popo6}
&\frac{1}{2}\|\varphi_{\beta, \ep, \lambda}(t)\|_{H}^2 
+ \int_{\Omega}G_{\ep}(\varphi_{\beta, \ep, \lambda}(t)) 
\\ \notag
&\geq \frac{1}{2}(1-\|\pi'\|_{L^{\infty}(\mathbb{R})})
                         \|\varphi_{\beta, \ep, \lambda}(t)\|_{H}^2  
       - 2\|\pi'\|_{L^{\infty}(\mathbb{R})}\ep
                         \|\varphi_{\beta, \ep, \lambda}(t)\|_{H}^2 
\\ \notag
&\geq \frac{1}{4}(1-\|\pi'\|_{L^{\infty}(\mathbb{R})})
                         \|\varphi_{\beta, \ep, \lambda}(t)\|_{H}^2  
\end{align}
for all $t \in [0, T]$, $\ep \in (0, \ep_{1})$ and $\lambda>0$. 
Hence, 
combining 
\eqref{tool2}, \eqref{tool4}, Remark \ref{MYreg}, \eqref{popo5} and \eqref{popo6} 
leads to the inequality    
\begin{align*}
&\frac{\alpha}{2}\|\mu_{\beta, \ep, \lambda}(t)\|_{H}^2 
+ \beta\int_{0}^{t}\|\partial_{t}\varphi_{\beta, \ep, \lambda}(s)\|_{H}^2\,ds  
+ \frac{1}{4}(1-\|\pi'\|_{L^{\infty}(\mathbb{R})})
                         \|\varphi_{\beta, \ep, \lambda}(t)\|_{H}^2  
\\ 
&+\frac{1}{2}\|\sigma_{\beta, \ep, \lambda}(t)\|_{H}^2 
+ p\int_{0}^{t}
            \|\sigma_{\beta, \ep, \lambda}(s)-\mu_{\beta, \ep, \lambda}(s)\|_{H}^2\,ds 
\\ 
&\leq \frac{\alpha}{2}\|\mu_{0}\|_{H}^2 + \frac{1}{2}\|\varphi_{0}\|_{V}^2  
         + \frac{1}{2}\|\varphi_{0}\|_{H}^2 
         + \int_{\Omega}G(\varphi_{0})
         + \frac{1}{2}\|\sigma_{0}\|_{H}^2
\end{align*}
for all $t \in [0, T]$, $\ep \in (0, \ep_{1})$ and $\lambda>0$. 
Thus there exists a constant $C_{1}=C_{1}(\alpha, \beta)>0$ such that 
\begin{align}\label{popo8}
&\|\mu_{\beta, \ep, \lambda}(t)\|_{H}^2 
  + \|\varphi_{\beta, \ep, \lambda}(t)\|_{H}^2 
  + \|\sigma_{\beta, \ep, \lambda}(t)\|_{H}^2 
  \\ \notag 
  &+ \int_{0}^{t}\|\partial_{t}\varphi_{\beta, \ep, \lambda}(s)\|_{H}^2\,ds 
  + p\int_{0}^{t}
            \|\sigma_{\beta, \ep, \lambda}(s)-\mu_{\beta, \ep, \lambda}(s)\|_{H}^2\,ds 
\leq C_{1}
\end{align}
for all $t \in [0, T]$, $\ep \in (0, \ep_{1})$ and $\lambda>0$. 
The third equation in \ref{Pbeplam} and the Young inequality yield that 
\begin{align*}
&\|\partial_{t}\sigma_{\beta, \ep, \lambda}(t)\|_{H}^2 
+ \frac{1}{2}\frac{d}{dt}
            \|(-\Delta)_{\lambda}^{1/2}\sigma_{\beta, \ep, \lambda}(t)\|_{H}^2 
\\ 
&= -p(\sigma_{\beta, \ep, \lambda}(t)-\mu_{\beta, \ep, \lambda}(t), 
                                           \partial_{t}\sigma_{\beta, \ep, \lambda}(t))_{H}
\\ 
&\leq 
   \frac{1}{2}\|\partial_{t}\sigma_{\beta, \ep, \lambda}(t)\|_{H}^2 
   + \frac{p^2}{2}
             \|\sigma_{\beta, \ep, \lambda}(t)-\mu_{\beta, \ep, \lambda}(t)\|_{H}^2 
\end{align*}
and then it follows from \eqref{tool4} and \eqref{popo8} that 
for all $\ep \in (0, \ep_{1})$ 
there exists a constant $C_{2}=C_{2}(\alpha, \beta, \ep)>0$ such that 
\begin{align}\label{popo9} 
\int_{0}^{t}\|\partial_{t}\sigma_{\beta, \ep, \lambda}(s)\|_{H}^2\,ds 
\leq C_2
\end{align}
for all $t \in [0, T]$ and all $\lambda>0$.   
On the other hand, we derive 
from the first equation in \ref{Pbeplam} and the Young inequality 
that 
\begin{align*}
&\alpha\|\partial_{t}\mu_{\beta, \ep, \lambda}(t)\|_{H}^2 
+ \frac{1}{2}\frac{d}{dt}
            \|(-\Delta)_{\lambda}^{1/2}\mu_{\beta, \ep, \lambda}(t)\|_{H}^2 
\\ 
&= p(\sigma_{\beta, \ep, \lambda}(t)-\mu_{\beta, \ep, \lambda}(t), 
                                  \partial_{t}\mu_{\beta, \ep, \lambda}(t))_{H} 
     - (\partial_{t}\varphi_{\beta, \ep, \lambda}(t), 
                                  \partial_{t}\mu_{\beta, \ep, \lambda}(t))_{H} 
\\ 
&\leq \frac{\alpha}{2}\|\partial_{t}\mu_{\beta, \ep, \lambda}(t)\|_{H}^2 
        + C_3(p\|\sigma_{\beta, \ep, \lambda}(t)-\mu_{\beta, \ep, \lambda}(t)\|_{H}^2 
                 + \|\partial_{t}\varphi_{\beta, \ep, \lambda}(t)\|_{H}^2)
\end{align*}
with some constant $C_{3}=C_{3}(\alpha)>0$, 
and hence from \eqref{tool4} and \eqref{popo8} we have 
\begin{align}\label{popo10}
\int_{0}^{t}\|\partial_{t}\mu_{\beta, \ep, \lambda}(s)\|_{H}^2\,ds 
\leq C_{4}
\end{align}
for all $t \in [0, T]$, $\ep \in (0, \ep_{1})$ and $\lambda>0$ 
with some constant $C_{4}=C_{4}(\alpha, \beta, \ep)>0$.

Therefore we see from \eqref{popo8}-\eqref{popo10} and  
each equations in \ref{Pbeplam} that Lemma \ref{esPbeplam} holds. 
\end{proof}
To establish existence of weak solutions to \ref{Pbeta} 
we consider the approximation  
\begin{equation}\label{Pbep}\tag*{(P)$_{\beta, \ep}$}
  \begin{cases} 
    \alpha\partial_{t} \mu_{\beta, \ep} 
         + \partial_{t} \varphi_{\beta, \ep}  
         -\Delta\mu_{\beta, \ep} 
         = p(\sigma_{\beta, \ep} 
              - \mu_{\beta, \ep})  
         & \mbox{in}\ \Omega\times(0, T),
 \\[1mm]
         \mu_{\beta, \ep} 
         = \beta\partial_{t}\varphi_{\beta, \ep} 
            + (-\Delta+1)\varphi_{\beta, \ep} 
            + G_{\ep}'(\varphi_{\beta, \ep}) 
         & \mbox{in}\ \Omega\times(0, T),
 \\[1mm]
         \partial_{t} \sigma_{\beta, \ep}  
         -\Delta\sigma_{\beta, \ep} 
         = -p(\sigma_{\beta, \ep} 
                - \mu_{\beta, \ep})   
         & \mbox{in}\ \Omega\times(0, T), 
 \\[1mm]
         \partial_{\nu}\mu_{\beta, \ep} 
         = \partial_{\nu}\varphi_{\beta, \ep} = \partial_{\nu}\sigma_{\beta, \ep} = 0                                    
         & \mbox{on}\ \partial\Omega\times(0, T),  
 \\[0.5mm]
        \mu_{\beta, \ep}(0) = \mu_{0\ep}, 
        \varphi_{\beta, \ep}(0) = \varphi_{0}, 
        \sigma_{\beta, \ep}(0) = \sigma_{0\ep}                                          
         & \mbox{in}\ \Omega.
  \end{cases}
\end{equation}
Then this paper defines weak solutions of \ref{Pbep} as follows.
%
%
%
%
 \begin{df}\label{defsolPbep}
 A triplet   
 $(\mu_{\beta, \ep}, \varphi_{\beta, \ep}, 
                                                       \sigma_{\beta, \ep})$ 
with 
    \begin{align*}
    &\mu_{\beta, \ep}, \varphi_{\beta, \ep}, \sigma_{\beta, \ep} 
          \in H^1(0, T; H) \cap L^2(0, T; W) 
    \end{align*}
 is called a {\it weak solution} of \ref{Pbep} if 
 $(\mu_{\beta, \ep}, \varphi_{\beta, \ep}, 
                                                       \sigma_{\beta, \ep})$ 
 satisfies 
    \begin{align}
        & \alpha(\partial_{t}\mu_{\beta, \ep}, v)_{H}
          + \bigl(\partial_{t}\varphi_{\beta, \ep}, v\bigr)_{H}
          + \bigl(\nabla\mu_{\beta, \ep}, \nabla v\bigr)_{H} 
          = 
          p(\sigma_{\beta, \ep} - \mu_{\beta, \ep}, v)_{H} 
         \label{defsolPbep1} \\ \notag 
         &\hspace{9cm}\mbox{a.e.\ on}\ (0, T)\quad  \mbox{for all }\ v\in V,  
     \\[2mm]
        & \mu_{\beta, \ep} = 
         \beta\partial_{t}\varphi_{\beta, \ep} 
         + (-\Delta + 1)\varphi_{\beta, \ep} 
         + G_{\ep}'(\varphi_{\beta, \ep})  \quad  
         \mbox{a.e.\ on}\ \Omega\times(0, T), \label{defsolPbep2}
     \\[2mm]
         &  (\partial_{t}\sigma_{\beta, \ep}, v)_{H} 
            + \bigl(\nabla\sigma_{\beta, \ep}, \nabla v\bigr)_{H} 
          = 
          -p(\sigma_{\beta, \ep}-\mu_{\beta, \ep}, v)_{H} 
          \label{defsolPbep3} 
      \\ \notag 
         &\hspace{9cm}\mbox{a.e.\ on}\ (0, T)\quad  \mbox{for all }\ v\in V,  
     \\[2mm]
         & \mu_{\beta, \ep}(0) = \mu_{0\ep},\ 
            \varphi_{\beta, \ep}(0) = \varphi_{0},\ 
        \sigma_{\beta, \ep}(0) = \sigma_{0\ep} 
                                                  \quad \mbox{a.e.\ on}\ \Omega. 
     \label{defsolPbep4}
     \end{align}
 \end{df}
\begin{lem}\label{solPbep}
Let $\ep_{1}$ be as in Lemma \ref{esPbeplam}. 
Then for all $\alpha>0$, $\beta>0$ and $\ep \in (0, \ep_{1})$     
there exists a weak solution 
$(\mu_{\beta, \ep}, \varphi_{\beta, \ep}, 
                                                       \sigma_{\beta, \ep})$ 
of {\rm \ref{Pbep}}. 
\end{lem}
\begin{proof}
Let $\alpha, \beta>0$ and let $\ep\in(0, \ep_{1})$.    
Then the estimates \eqref{esPbeplam1}-\eqref{esPbeplam4} yield that 
there exist some functions 
$\mu_{\beta, \ep}, \varphi_{\beta, \ep}, \sigma_{\beta, \ep} 
         \in H^1(0, T; H) \cap L^2(0, T; W)$ satisfying  
\begin{align}
&\mu_{\beta, \ep, \lambda} \to \mu_{\beta, \ep}
\quad \mbox{weakly$^{*}$ in}\ L^{\infty}(0, T; H), 
\label{Kuri1}\\ 
&\varphi_{\beta, \ep, \lambda} \to \varphi_{\beta, \ep}
\quad \mbox{weakly$^{*}$ in}\ L^{\infty}(0, T; H), 
\label{Kuri2}\\
&\sigma_{\beta, \ep, \lambda} \to \sigma_{\beta, \ep}
\quad \mbox{weakly$^{*}$ in}\ L^{\infty}(0, T; H), 
\label{Kuri3}\\
&\partial_{t}\varphi_{\beta, \ep, \lambda} \to 
\partial_{t}\varphi_{\beta, \ep}  
\quad \mbox{weakly in}\ L^2(0, T; H), 
\label{Kuri4}\\
&\partial_{t}\sigma_{\beta, \ep, \lambda} \to 
\partial_{t}\sigma_{\beta, \ep}  
\quad \mbox{weakly in}\ L^2(0, T; H), 
\label{Kuri5}\\
&\partial_{t}\mu_{\beta, \ep, \lambda} \to 
\partial_{t}\mu_{\beta, \ep}  
\quad \mbox{weakly in}\ L^2(0, T; H), 
\label{Kuri6}\\
&(-\Delta)_{\lambda}\mu_{\beta, \ep, \lambda}  \to 
-\Delta \mu_{\beta, \ep}  \quad \mbox{weakly in}\ L^2(0, T; H), 
\label{Kuri7}\\
&(-\Delta)_{\lambda}\varphi_{\beta, \ep, \lambda}  \to 
-\Delta \varphi_{\beta, \ep}  \quad \mbox{weakly in}\ L^2(0, T; H), 
\label{Kuri8}\\
&(-\Delta)_{\lambda}\sigma_{\beta, \ep, \lambda}  \to 
-\Delta \sigma_{\beta, \ep}  \quad \mbox{weakly in}\ L^2(0, T; H) 
\label{Kuri9}
\end{align}
as $\lambda = \lambda_j \searrow0$. 
We can obtain \eqref{defsolPbep1} 
by \eqref{Kuri1}, \eqref{Kuri3}, \eqref{Kuri4}, \eqref{Kuri6} and \eqref{Kuri7}. 
Now we show \eqref{defsolPbep2} and \eqref{defsolPbep3}. 
To verify \eqref{defsolPbep2} it suffices to confirm that   
\begin{align}\label{henbun}
&\int_0^T \Bigl(\int_{\Omega} 
     \bigl(
     \mu_{\beta, \ep}(t)-\beta\partial_{t}\varphi_{\beta, \ep}(t)
     -(-\Delta+1)\varphi_{\beta, \ep}(t) 
\\ \notag
     &\hspace{70mm}-G_{\ep}'(\varphi_{\beta, \ep}(t)) 
     \bigr)\psi(t) \Bigr)\,dt 
= 0 
\end{align}
for all $\psi \in C_{\mathrm c}^{\infty}([0, T] \times \overline{\Omega})$. 
From the second equation in \ref{Pbeplam} 
we infer that 
\begin{align}\label{cute1}
&0=
\int_{0}^{T} 
\bigl(\mu_{\beta, \ep, \lambda}(t)
     -\beta\partial_{t}\varphi_{\beta, \ep, \lambda}(t)
     -((-\Delta)_{\lambda}+1)\varphi_{\beta, \ep, \lambda}(t), 
                                                                                 \psi(t)\bigr)_{H}\,dt 
\\ \notag 
&\,\quad
- \int_{0}^{T} (G_{\ep}'(\varphi_{\beta, \ep, \lambda}(t)), 
                                                                                    \psi(t))_{H}\,dt.
\end{align} 
Here there exists a bounded domain $D\subset\Omega$ with smooth boundary 
such that 
\begin{align*}
\mbox{supp}\,\psi \subset D\times(0, T).   
\end{align*}
It follows from \eqref{esPbeplam1}, \eqref{esPbeplam4} 
and Lemma \ref{keylemma2} 
that 
\begin{align}\label{tuyoi}
\varphi_{\beta, \ep, \lambda} \to \varphi_{\beta, \ep} 
\quad \mbox{in}\ 
L^2(0, T; L^2(D))
\end{align}
as $\lambda = \lambda_j \searrow0$. 
Since $G_{\ep}'=B_{\ep}+\pi$ is Lipschitz continuous, 
we see from \eqref{tuyoi} that 
\begin{align}\label{cute2}
\int_{0}^{T} (G_{\ep}'(\varphi_{\beta, \ep, \lambda}(t)), \psi(t))_{H}\,dt 
&= \int_0^T 
     \Bigl(\int_{D} G_{\ep}'(\varphi_{\beta, \ep, \lambda}(t))\psi(t) \Bigr)\,dt 
\\ \notag 
&\to \int_0^T 
             \Bigl(\int_{D} 
               G_{\ep}'(\varphi_{\beta, \ep}(t))\psi(t) \Bigr)\,dt  
\\ \notag
&= \int_{0}^{T} (G_{\ep}'(\varphi_{\beta, \ep}(t)), \psi(t))_{H}\,dt
\end{align}
as $\lambda = \lambda_j \searrow0$. 
Thus \eqref{Kuri1}, \eqref{Kuri2}, \eqref{Kuri4}, \eqref{Kuri8}, 
\eqref{cute1} and \eqref{cute2} lead to \eqref{henbun}. 
On the other hand, 
combining \eqref{Kuri1}, \eqref{Kuri3}, \eqref{Kuri5} and \eqref{Kuri9} 
leads to \eqref{defsolPbep3}. 

Next we prove \eqref{defsolPbep4}. 
Let $E \subset \Omega$ be 
an arbitrary bounded domain with smooth boundary.  
Then from \eqref{esPbeplam1}, \eqref{esPbeplam3} 
and Lemma \ref{keylemma} we have 
\begin{equation}\label{itiyou}
J^{1/2}_1\mu_{\beta, \ep, \lambda} \to 
J^{1/2}_1\mu_{\beta, \ep} 
\quad \mbox{in}\ C([0, T]; L^2(E))
\end{equation}
as $\lambda = \lambda_j \searrow0$. 
Therefore, 
since $\mu_{\beta, \ep, \lambda}(0)=\mu_{0\ep}$, 
\eqref{itiyou} yields that 
$$
J^{1/2}_1\mu_{\beta, \ep}(0)=J^{1/2}_1\mu_{0\ep} 
\quad \mbox{a.e.\ in}\ E.
$$
Because $E \subset \Omega$ is arbitrary, we conclude that 
$$
J^{1/2}_1\mu_{\beta, \ep}(0)=J^{1/2}_1\mu_{0\ep} 
\quad \mbox{a.e.\ in}\ \Omega.
$$ 
Hence, since $J^{1/2}_1\mu_{0\ep} \in H$, we see that 
$$
J^{1/2}_1\mu_{\beta, \ep}(0)=J^{1/2}_1\mu_{0\ep}  
\quad \mbox{in}\ H, 
$$
that is, it holds that 
$$
\mu_{\beta, \ep}(0)=\mu_{0\ep} 
\quad \mbox{in}\ H. 
$$ 
Similarly, we can prove that 
$$
\varphi_{\beta, \ep}(0)=\varphi_{0}, \quad 
\sigma_{\beta, \ep}(0)=\sigma_{0\ep}  
\qquad \mbox{in}\ H. 
$$
Thus \eqref{defsolPbep4} holds. 
\end{proof}
%
\begin{lem}\label{uniquesolPbep} 
Let $\ep_{1}$ be as in Lemma \ref{esPbeplam}. 
Then there exists $\alpha_{1} \in (0, 1)$ such that 
for all $\alpha \in (0, \alpha_{1})$, $\beta>0$ and $\ep \in (0, \ep_{1})$ 
the weak solution 
$(\mu_{\beta, \ep}, \varphi_{\beta, \ep}, 
                                                       \sigma_{\beta, \ep})$ 
of {\rm \ref{Pbep}} is unique. 
\end{lem}
\begin{proof}
We can show this lemma
in reference to \cite[Proof of Theorem 2.3]{CGRS-2017}. 
Let $(\mu_{j}, \varphi_{j}, \sigma_{j})$, $j=1, 2$, be two weak solutions of 
\ref{Pbep}. 
We put $\overline{\mu}:=\mu_{1}-\mu_{2}, 
\overline{\varphi}:=\varphi_{1}-\varphi_{2}, 
\overline{\sigma}:=\sigma_{1}-\sigma_{2}, 
\theta:=\alpha\overline{\mu}+\overline{\varphi}+\overline{\sigma}$ and 
$\overline{R}:=R_{1}-R_{2}$ ($R_{j}:=p(\sigma_{j}-\mu_{j})$, $j=1, 2$). 
Then we derive from 
\eqref{defF}, \eqref{innerVstar}, \eqref{defsolPbep1} and \eqref{defsolPbep3} 
that  
\begin{align*}
\frac{1}{2}\frac{d}{dt}\|\theta(t)\|_{V^{*}}^2 
&= \langle \partial_{t}\theta(t), 
                             F^{-1}\theta(t) \rangle_{V^{*}, V} 
\\ 
&=-\bigl(\nabla(\overline{\mu}(t) + \overline{\sigma}(t)), 
                                        \nabla F^{-1}\theta(t)\bigr)_{H} 
\\ 
&=-(F^{-1}\theta(t), \overline{\mu}(t) + \overline{\sigma}(t))_{V} 
    + \langle \overline{\mu}(t) + \overline{\sigma}(t),  
                                                    F^{-1}\theta(t) \rangle_{V^{*}, V} 
\\
&= -(\theta(t), \overline{\mu}(t) + \overline{\sigma}(t))_{H} 
    + (\overline{\mu}(t) + \overline{\sigma}(t), \theta(t))_{V^{*}},  
\end{align*}
and hence we can obtain that 
\begin{align}\label{toto1}
\frac{1}{2}\|\theta(t)\|_{V^{*}}^2 
+ \int_{0}^{t}(\theta(s), \overline{\mu}(s) + \overline{\sigma}(s))_{H}\,ds  
= \int_{0}^{t}(\overline{\mu}(s) + \overline{\sigma}(s), \theta(s))_{V^{*}}\,ds.
\end{align}
On the other hand, we infer 
from \eqref{defsolPbep2} and \eqref{defsolPbep3} that 
\begin{align}\label{toto2}
&\int_{0}^{t}\|\overline{\varphi}(s)\|_{V}^2\,ds 
+ \int_{0}^{t}(B_{\ep}(\varphi_{1}(s))-B_{\ep}(\varphi_{2}(s)), 
                                                    \overline{\varphi}(s))_{H}\,ds 
-\int_{0}^{t} (\overline{\mu}(s), \overline{\varphi}(s))_{H}\,ds 
\\ \notag 
&= -\frac{\beta}{2}\|\overline{\varphi}(t)\|_{H}^2 
    -\int_{0}^{t}(\pi(\varphi_{1}(s))-\pi(\varphi_{2}(s)), 
                                                      \overline{\varphi}(s))_{H}\,ds  
\end{align}
and 
\begin{align}\label{toto3}
\frac{1}{2}\|\overline{\sigma}(t)\|_{H}^2 
+ \int_{0}^{t}\|\overline{\sigma}(s)\|_{V}^2\,ds 
= -\int_{0}^{t} (\overline{R}(s), \overline{\sigma}(s))\,ds 
    + \int_{0}^{t}\|\overline{\sigma}(s)\|_{H}^2\,ds. 
\end{align}
Thus, combining \eqref{toto1}-\eqref{toto3}, we have 
\begin{align}\label{toto4}
&\frac{1}{2}\|\theta(t)\|_{V^{*}}^2 
+ \int_{0}^{t}(\theta(s), \overline{\mu}(s) + \overline{\sigma}(s))_{H}\,ds 
+ \int_{0}^{t}\|\overline{\varphi}(s)\|_{V}^2\,ds 
\\ \notag 
&+ \int_{0}^{t}(B_{\ep}(\varphi_{1}(s))-B_{\ep}(\varphi_{2}(s)), 
                                                             \overline{\varphi}(s))_{H}\,ds 
-\int_{0}^{t} (\overline{\mu}(s), \overline{\varphi}(s))_{H}\,ds 
\\ \notag 
&+ \frac{1}{2}\|\overline{\sigma}(t)\|_{H}^2 
+ \int_{0}^{t}\|\overline{\sigma}(s)\|_{V}^2\,ds 
\\ \notag 
&= \int_{0}^{t}(\overline{\mu}(s) + \overline{\sigma}(s), \theta(s))_{V^{*}}\,ds 
-\frac{\beta}{2}\|\overline{\varphi}(t)\|_{H}^2 
-\int_{0}^{t}(\pi(\varphi_{1}(s))-\pi(\varphi_{2}(s)), 
                                                                  \overline{\varphi}(s))_{H}\,ds  
\\ \notag 
&\,\quad -\int_{0}^{t} (\overline{R}(s), \overline{\sigma}(s))\,ds 
    + \int_{0}^{t}\|\overline{\sigma}(s)\|_{H}^2\,ds. 
\end{align} 
Here we derive from the Young inequality that 
for all $\alpha>0$ and all $\delta>0$ there exist 
constants 
$C_{1}=C_{1}(\alpha)>0$, $C_{2}=C_{2}(\delta)>0$, $C_{3}=C_{3}(\alpha)>0$, 
$C_{4}=C_{4}(\alpha) > 0$ and $C_{5}=C_{5}(\delta) > 0$ 
such that 
\begin{align}\label{toto5}
&\int_{0}^{t}(\overline{\mu}(s) + \overline{\sigma}(s), \theta(s))_{V^{*}}\,ds  
\\ \notag 
&\leq \frac{\alpha}{8}\int_{0}^{t} \|\overline{\mu}(s)\|_{H}^2\,ds 
      + \frac{\alpha}{8}\int_{0}^{t} \|\overline{\sigma}(s)\|_{H}^2\,ds 
      + C_{1}\int_{0}^{t} \|\theta(s)\|_{V^{*}}^2\,ds, \\[4mm] \label{toto6}
&-\int_{0}^{t}(\pi(\varphi_{1}(s))-\pi(\varphi_{2}(s)), \overline{\varphi}(s))_{H}\,ds 
\\ \notag 
&\leq \|\pi'\|_{L^{\infty}(\mathbb{R})}\int_{0}^{t}\|\overline{\varphi}(s)\|_{H}^2\,ds 
\\ \notag 
&= \|\pi'\|_{L^{\infty}(\mathbb{R})}
     \int_{0}^{t} \bigl( 
                   \langle \theta(s), \overline{\varphi}(s) \rangle_{V^{*}, V} 
                   - \alpha(\overline{\varphi}(s), \overline{\mu}(s))_{H} 
                   - (\overline{\varphi}(s), \overline{\sigma}(s))_{H} 
                   \bigr)\,ds 
\\ \notag 
&\leq \delta\int_{0}^{t}\|\overline{\varphi}(s)\|_{V}^2\,ds 
         + C_2\int_{0}^{t}\|\theta(s)\|_{V^{*}}^2\,ds 
         + \frac{\alpha}{8}\int_{0}^{t}\|\overline{\mu}(s)\|_{H}^2\,ds 
\\ \notag 
   &\,\quad + 2\|\pi'\|_{L^{\infty}(\mathbb{R})}^2\alpha
                                  \int_{0}^{t}\|\overline{\varphi}(s)\|_{V}^2\,ds 
        + C_{2}\int_{0}^{t}\|\overline{\sigma}(s)\|_{H}^2\,ds,  \\[4mm] \label{toto7}
&-\int_{0}^{t} (\overline{R}(s), \overline{\sigma}(s))\,ds 
\\ \notag  
&= -p\int_{0}^{t}\|\overline{\sigma}(s)\|_{H}^2\,ds 
  + p\int_{0}^{t} (\overline{\mu}(s), \overline{\sigma}(s))\,ds  
\\ \notag 
&\leq \frac{\alpha}{8}\int_{0}^{t}\|\overline{\mu}(s)\|_{H}^2\,ds 
         + C_{3}\int_{0}^{t}\|\overline{\sigma}(s)\|_{H}^2\,ds, \\[4mm] \label{toto8}
&\int_{0}^{t}(\theta(s), \overline{\mu}(s) + \overline{\sigma}(s))_{H}\,ds 
-\int_{0}^{t} (\overline{\mu}(s), \overline{\varphi}(s))_{H}\,ds  
\\ \notag 
&= \alpha\int_{0}^{t}\|\overline{\mu}(s)\|_{H}^2\,ds  
    + \int_{0}^{t} (\overline{\mu}(s), \overline{\sigma}(s))_{H}\,ds 
    + \int_{0}^{t} (\overline{\sigma}(s), \theta(s))_{H}\,ds 
\\ \notag 
&\geq \alpha\int_{0}^{t}\|\overline{\mu}(s)\|_{H}^2\,ds  
    + \int_{0}^{t} (\overline{\mu}(s), \overline{\sigma}(s))_{H}\,ds 
    - \int_{0}^{t}\|\overline{\sigma}(s)\|_{V}\|\theta(s)\|_{V^{*}}\,ds  
\\ \notag 
&\geq \frac{\alpha}{2}\int_{0}^{t}\|\overline{\mu}(s)\|_{H}^2\,ds 
         - C_{4}\int_{0}^{t}\|\overline{\sigma}(s)\|_{H}^2\,ds   
         - \delta\int_{0}^{t}\|\overline{\sigma}(s)\|_{V}^2\,ds 
         - C_{5}\int_{0}^{t}\|\theta(s)\|_{V^{*}}^2\,ds. 
\end{align} 
Thus it follows from \eqref{toto4}-\eqref{toto8} that 
for all $\alpha>0$ and all $\delta>0$  
there exists a constant $C_{6}=C_{6}(\alpha, \delta)>0$ such that 
\begin{align*}
&\frac{1}{2}\|\theta(t)\|_{V^{*}}^2 
+ \frac{\alpha}{8}\int_{0}^{t}\|\overline{\mu}(s)\|_{H}^2\,ds 
+ (1-2\|\pi'\|_{L^{\infty}(\mathbb{R})}^2\alpha-2\delta)
                                       \int_{0}^{t}\|\overline{\varphi}(s)\|_{V}^2\,ds 
\\
&+ \frac{1}{2}\|\overline{\sigma}(t)\|_{H}^2 
  + (1-\delta)\int_{0}^{t}\|\overline{\sigma}(s)\|_{V}^2\,ds 
\\  
&\leq C_{6}\int_{0}^{t}\|\overline{\sigma}(s)\|_{H}^2\,ds 
         + C_{6}\int_{0}^{t}\|\theta(s)\|_{V^{*}}^2\,ds.  
\end{align*}
Therefore, choosing $\delta > 0$ and $\alpha>0$ small enough 
and applying the Gronwall lemma, 
we can prove that 
there exists $\alpha_{1}\in(0, 1)$ such that 
for all $\alpha\in(0, \alpha_{1})$, $\beta>0$ and $\ep\in(0, \ep_{1})$  
the weak solution of \ref{Pbep} is unique. 
\end{proof}
We can establish estimates of solutions for \ref{Pbep} 
in reference to 
\cite[4  A priori estimates]{CGH-2015} and 
\cite[3.\ Proofs of Theorems 2.2 and 2.3]{CGRS-2015}. 
%
\begin{lem}\label{estimatesolPbep1} 
Let $\ep_{1}$ be as in Lemma \ref{esPbeplam} 
and let $\alpha_{1}$ be as in Lemma \ref{uniquesolPbep}. 
Then there exists a constant $C=C(T)>0$ such that 
for all $\alpha \in (0, \alpha_{1})$, $\beta>0$ and $\ep \in (0, \ep_{1})$ 
the weak solution 
$(\mu_{\beta, \ep}, \varphi_{\beta, \ep}, 
                                                       \sigma_{\beta, \ep})$ 
of {\rm \ref{Pbep}}  
satisfies the inequality 
\begin{align*}
&\alpha^{1/2}\|\mu_{\beta, \ep}\|_{L^{\infty}(0, T; H)} 
       + \|\nabla \mu_{\beta, \ep}\|_{L^2(0, T; H)} 
       + \beta^{1/2}\|\partial_{t}\varphi_{\beta, \ep}\|_{L^2(0, T; H)} 
\\
     &+ \|\varphi_{\beta, \ep}\|_{L^{\infty}(0, T; V)} 
     + \|\sigma_{\beta, \ep}\|_{H^1(0, T; V^{*}) \cap L^{\infty}(0, T; H) 
                                                                             \cap L^2(0, T; V)}  
\\
&+ \|\partial_{t}(\alpha\mu_{\beta, \ep}+\varphi_{\beta, \ep})\|_{L^2(0, T; V^{*})} 
+ p^{1/2}\|\sigma_{\beta, \ep}-\mu_{\beta, \ep}\|_{L^2(0, T; H)}  
     \\ \notag 
     &\leq C\Bigl(\alpha^{1/2}\|\mu_0\|_{H} + \|\varphi_0\|_{V} 
                     + \|G(\varphi_0)\|_{L^1(\Omega)}^{1/2} 
                     + \|\sigma_0\|_{H} \Bigr).   
\end{align*}
\end{lem}
\begin{proof}
Let $\alpha \in (0, \alpha_{1})$, let $\beta>0$ and let $\ep \in (0, \ep_{1})$.  
Then we see from \eqref{defsolPbep1}-\eqref{defsolPbep3} that 
\begin{align*}
&\frac{\alpha}{2}\frac{d}{dt}\|\mu_{\beta, \ep}(t)\|_{H}^2 
+ (\partial_{t}\varphi_{\beta, \ep}(t), \mu_{\beta, \ep}(t))_{H} 
+ \|\nabla \mu_{\beta, \ep}(t)\|_{H}^2 
= (R_{\beta, \ep}(t), \mu_{\beta, \ep}(t))_{H},  \\ 
&(\partial_{t}\varphi_{\beta, \ep}(t), \mu_{\beta, \ep}(t))_{H} 
= \beta\|\partial_{t}\varphi_{\beta, \ep}(t)\|_{H}^2 
   + \frac{1}{2}\frac{d}{dt}\|\varphi_{\beta, \ep}(t)\|_{V}^2 
   + \frac{d}{dt}\int_{\Omega} G_{\ep}(\varphi_{\beta, \ep}(t)), \\    
&\frac{1}{2}\frac{d}{dt}\|\sigma_{\beta, \ep}(t)\|_{H}^2 
+ \|\nabla\sigma_{\beta, \ep}(t)\|_{H}^2 
= - (R_{\beta, \ep}(t), \sigma_{\beta, \ep}(t))_{H},   
\end{align*}
where $R_{\beta, \ep}:=p(\sigma_{\beta, \ep}-\mu_{\beta, \ep})$,   
and hence \eqref{tool2} and Remark \ref{MYreg} yield that 
\begin{align}\label{coco1}
&\frac{\alpha}{2}\|\mu_{\beta, \ep}(t)\|_{H}^2 
+ \int_{0}^{t}\|\nabla \mu_{\beta, \ep}(s)\|_{H}^2\,ds  \\ \notag
&+ \beta\int_{0}^{t}\|\partial_{t}\varphi_{\beta, \ep}(s)\|_{H}^2\,ds  
+ \frac{1}{2}\|\varphi_{\beta, \ep}(t)\|_{V}^2 
+ \int_{\Omega} G_{\ep}(\varphi_{\beta, \ep}(t)) \\ \notag 
&+ \frac{1}{2}\|\sigma_{\beta, \ep}(t)\|_{H}^2 
+ \int_{0}^{t}\|\nabla \sigma_{\beta, \ep}(s)\|_{H}^2\,ds  
+ p\int_{0}^{t}\|\sigma_{\beta, \ep}(s)-\mu_{\beta, \ep}(s)\|_{H}^2\,ds  \\ \notag
&\leq C_{1}\Bigl(\alpha\|\mu_0\|_{H}^2 + \|\varphi_0\|_{V}^2 
                     + \|G(\varphi_0)\|_{L^1(\Omega)} 
                     + \|\sigma_0\|_{H}^2 \Bigr)   
\end{align} 
with some constant $C_{1}>0$. 
Here we derive from \eqref{popo6} that  
\begin{align}\label{coco2}
&\frac{1}{2}\|\varphi_{\beta, \ep}(t)\|_{V}^2 
+ \int_{\Omega}G_{\ep}(\varphi_{\beta, \ep}(t)) 
\\ \notag 
&\geq \frac{1}{2}\|\nabla\varphi_{\beta, \ep}(t)\|_{H}^2 
       + \frac{1}{4}(1-\|\pi'\|_{L^{\infty}(\mathbb{R})})
                                                             \|\varphi_{\beta, \ep}(t)\|_{H}^2 
\\ \notag 
&\geq \frac{1}{4}(1-\|\pi'\|_{L^{\infty}(\mathbb{R})})\|\varphi_{\beta, \ep}(t)\|_{V}^2. 
\end{align} 
Hence, by virtue of \eqref{coco1} and \eqref{coco2}, we can prove that 
\begin{align}\label{coco3}
&\frac{\alpha}{2}\|\mu_{\beta, \ep}(t)\|_{H}^2 
+ \int_{0}^{t}\|\nabla \mu_{\beta, \ep}(s)\|_{H}^2\,ds  \\ \notag
&+ \beta\int_{0}^{t}\|\partial_{t}\varphi_{\beta, \ep}(s)\|_{H}^2\,ds  
+ \frac{1}{4}(1-\|\pi'\|_{L^{\infty}(\mathbb{R})})\|\varphi_{\beta, \ep}(t)\|_{V}^2 
\\ \notag 
&+ \frac{1}{2}\|\sigma_{\beta, \ep}(t)\|_{H}^2 
+ \int_{0}^{t}\|\nabla \sigma_{\beta, \ep}(s)\|_{H}^2\,ds  
+ p\int_{0}^{t}\|\sigma_{\beta, \ep}(s)-\mu_{\beta, \ep}(s)\|_{H}^2\,ds   \\ \notag
&\leq C_{1}\Bigl(\alpha\|\mu_0\|_{H}^2 + \|\varphi_0\|_{V}^2 
                     + \|G(\varphi_0)\|_{L^1(\Omega)} 
                     + \|\sigma_0\|_{H}^2 \Bigr).    
\end{align}
On the other hand, 
from \eqref{defsolPbep3} we have 
\begin{align}\label{coco5}
&\left|\int_{0}^{T} 
         \langle \partial_{t}\sigma_{\beta, \ep}(t), v(t) \rangle_{V^{*}, V}\,dt \right| 
\\ \notag 
&\leq (\|\nabla\sigma_{\beta, \ep}\|_{L^2(0, T; H)} + 
          p\|\sigma_{\beta, \ep}-\mu_{\beta, \ep}\|_{L^2(0, T; H)})
                                                                                  \|v\|_{L^2(0, T; V)} 
\end{align}
for all $v \in L^2(0, T; V)$ and  
we infer from 
\eqref{defsolPbep1} and \eqref{defsolPbep3} that 
\begin{align}\label{coco6}
&\left|\int_{0}^{T} 
         \langle \partial_{t}(\alpha\mu_{\beta, \ep}+\varphi_{\beta, \ep})(t), 
                                                                    v(t) \rangle_{V^{*}, V}\,dt \right| 
\\ \notag
&\leq \left|\int_{0}^{T} 
         \langle \partial_{t}\sigma_{\beta, \ep}(t), v(t) \rangle_{V^{*}, V}\,dt \right| 
      + \|\nabla(\mu_{\beta, \ep}+\sigma_{\beta, \ep})\|_{L^2(0, T; H)}
                                                                             \|v\|_{L^2(0, T; V)} 
\end{align} 
for all $v \in L^2(0, T; V)$.  
Thus, 
combining \eqref{coco3}-\eqref{coco6},   
we can obtain this lemma. 
\end{proof}
\begin{lem}\label{estimatesolPbep2} 
Let $\ep_{1}$ be as in Lemma \ref{esPbeplam} 
and let $\alpha_{1}$ be as in Lemma \ref{uniquesolPbep}. 
Then there exists a constant $C=C(T)>0$ such that 
for all $\alpha \in (0, \alpha_{1})$, $\beta \in (0, 1)$ and $\ep \in (0, \ep_{1})$ 
the weak solution 
$(\mu_{\beta, \ep}, \varphi_{\beta, \ep}, 
                                                       \sigma_{\beta, \ep})$ 
of {\rm \ref{Pbep}}  
satisfies the inequality 
\begin{align*}
&\|\mu_{\beta, \ep}\|_{L^2(0, T; V)} 
     + \|\varphi_{\beta, \ep}\|_{L^2(0, T; W)}  
     + \|B_{\ep}(\varphi_{\beta, \ep})\|_{L^2(0, T; H)} 
     \\ \notag 
&\leq C\Bigl(\alpha^{1/2}\|\mu_0\|_{H} + \|\varphi_0\|_{V} 
                     + \|G(\varphi_0)\|_{L^1(\Omega)}^{1/2} 
                     + \|\sigma_0\|_{H} + \|\mu_{\beta, \ep}\|_{L^2(0, T; H)}\Bigr).   
\end{align*}
\end{lem}
\begin{proof}
Let $\alpha \in (0, \alpha_{1})$, let $\beta \in (0, 1)$ and let $\ep \in (0, \ep_{1})$.  
Then it follows from \eqref{defsolPbep2} and the Young inequality that 
\begin{align*}
&\|(-\Delta+1)\varphi_{\beta, \ep}(t)\|_{H}^2 
\\ 
&= ((-\Delta+1)\varphi_{\beta, \ep}(t), 
           \mu_{\beta, \ep}(t)-\beta\partial_{t}\varphi_{\beta, \ep}(t)
                        -B_{\ep}(\varphi_{\beta, \ep}(t))-\pi(\varphi_{\beta, \ep}(t)))_{H} 
\\
&\leq ((-\Delta+1)\varphi_{\beta, \ep}(t), 
           \mu_{\beta, \ep}(t)-\beta\partial_{t}\varphi_{\beta, \ep}(t)
                                                      -\pi(\varphi_{\beta, \ep}(t)))_{H} 
\\ 
&\leq \frac{1}{2}\|(-\Delta+1)\varphi_{\beta, \ep}(t)\|_{H}^2 
        + \frac{3}{2}(\|\mu_{\beta, \ep}(t)\|_{H}^2 
                    + \beta\|\partial_{t}\varphi_{\beta, \ep}(t)\|_{H}^2 
                    + \|\pi'\|_{L^{\infty}(\mathbb{R})}^2\|\varphi_{\beta, \ep}(t)\|_{H}^2). 
\end{align*}
Hence, letting $C_{1}$ be a positive constant appearing in the elliptic estimate 
$\|w\|_{W} \leq C_{1}\|(-\Delta+1)w\|_{H}$ for all $w \in W$, 
we have   
\begin{align*}
\|\varphi_{\beta, \ep}\|_{L^2(0, T; W)}^2 
\leq 3C_{1}^2(\|\mu_{\beta, \ep}\|_{L^2(0, T; H)}^2  
                    + \beta\|\partial_{t}\varphi_{\beta, \ep}\|_{L^2(0, T; H)}^2 
                    + \|\pi'\|_{L^{\infty}(\mathbb{R})}^2
                                                   \|\varphi_{\beta, \ep}\|_{L^2(0, T; H)}^2). 
\end{align*}
On the other hand, we see from \eqref{defsolPbep2} and the Young inequality 
that 
\begin{align*}
\|B_{\ep}(\varphi_{\beta, \ep})\|_{L^2(0, T; H)}  
&\leq \|\mu_{\beta, \ep}\|_{L^2(0, T; H)} 
       + \beta^{1/2}\|\partial_{t}\varphi_{\beta, \ep}\|_{L^2(0, T; H)} 
\\ 
       &\,\quad+ \|(-\Delta+1)\varphi_{\beta, \ep}\|_{L^2(0, T; H)} 
       + \|\pi'\|_{L^{\infty}(\mathbb{R})}\|\varphi_{\beta, \ep}\|_{L^2(0, T; H)}. 
\end{align*}
Therefore we can prove this lemma by Lemma \ref{estimatesolPbep1}. 
\end{proof}
\begin{prth1.1}
Let $\alpha_{1}$ be as in Lemma \ref{uniquesolPbep}, 
let $\alpha \in (0, \alpha_{1})$ and let $\beta \in (0, 1)$. 
Then, by Lemmas \ref{estimatesolPbep1} and \ref{estimatesolPbep2}, 
there exist some functions 
\begin{align*}
&\mu_{\beta} \in H^1(0, T; V^{*}) \cap L^{\infty}(0, T; H) \cap L^2(0, T; V), 
\\
&\varphi_{\beta} \in H^1(0, T; H) \cap L^{\infty}(0, T; V) \cap L^2(0, T; W),
\\ 
&\sigma_{\beta} \in H^1(0, T; V^{*}) \cap L^{\infty}(0, T; H) \cap L^2(0, T; V), 
\\ 
&\xi_{\beta} \in L^2(0, T; H) 
\end{align*}
such that 
\begin{align}
&\mu_{\beta, \ep} \to \mu_{\beta}
\quad \mbox{weakly$^{*}$ in}\ L^{\infty}(0, T; H)\ 
\mbox{and weakly$^{*}$ in}\ L^2(0, T; V), 
\label{ume1}\\ \notag 
&\varphi_{\beta, \ep} \to \varphi_{\beta}
\quad \mbox{weakly$^{*}$ in}\ L^{\infty}(0, T; V), 
\\ 
&\sigma_{\beta, \ep} \to \sigma_{\beta}
\quad \mbox{weakly$^{*}$ in}\ L^{\infty}(0, T; H)\ 
\mbox{and weakly$^{*}$ in}\ L^2(0, T; V), 
\label{ume3}\\
&\partial_{t}\mu_{\beta, \ep} \to 
(\mu_{\beta})_{t}  
\quad \mbox{weakly in}\ L^2(0, T; V^{*}), 
\label{ume4}\\
&\partial_{t}\varphi_{\beta, \ep} \to 
\partial_{t}\varphi_{\beta}  
\quad \mbox{weakly in}\ L^2(0, T; H), 
\label{ume5}\\
&\partial_{t}\sigma_{\beta, \ep} \to 
(\sigma_{\beta})_{t}  
\quad \mbox{weakly in}\ L^2(0, T; V^{*}), 
\label{ume6}\\
&B_{\ep}(\varphi_{\beta, \ep})  \to \xi_{\beta} 
\quad \mbox{weakly in}\ L^2(0, T; H), 
\label{ume7}\\
&(-\Delta+1)\varphi_{\beta, \ep}  \to 
(-\Delta+1)\varphi_{\beta}  \quad \mbox{weakly in}\ L^2(0, T; H) 
\label{ume8}
\end{align}
as $\ep=\ep_{j}\searrow0$. 
From \eqref{defsolPbep1}, \eqref{defsolPbep3}, \eqref{ume1}-\eqref{ume6} 
we can obtain \eqref{defsolPbe1} and \eqref{defsolPbe3}. 
Also, 
from \eqref{ume1}, \eqref{ume5}, \eqref{ume7}, \eqref{ume8}, 
Lemmas \ref{keylemma2}, \ref{estimatesolPbep1} and \ref{estimatesolPbep2}   
we can prove that 
\begin{align}\label{semidefsolPbe2}
\mu_{\beta} 
= \beta\partial_{t}\varphi_{\beta} + (-\Delta+1)\varphi_{\beta} 
   + \xi_{\beta} + \pi(\varphi_{\beta}) 
\end{align}
in a similar way to the proof of Lemma \ref{solPbep}.   
Now we show \eqref{defsolPbe2} by proving that 
\begin{align}\label{xibetainB(varphibeta)}
\xi_{\beta} \in B(\varphi_{\beta})  \quad \mbox{a.e.\ on}\ \Omega\times(0, T). 
\end{align}
Let $E \subset \Omega$ be 
an arbitrary bounded domain with smooth boundary.  
Then from Lemmas \ref{keylemma2}, \ref{estimatesolPbep1} 
and \ref{estimatesolPbep2}  
we have 
\begin{align}\label{tuyoikyokusyo}
1_{E}\varphi_{\beta, \ep} \to 
1_{E}\varphi_{\beta} 
\quad \mbox{in}\ L^2(0, T; H)
\end{align}
as $\ep=\ep_{j}\searrow0$, 
where $1_{E}$ is the characteristic function of $E$.    
It follows from \eqref{ume7} and \eqref{tuyoikyokusyo} that 
\begin{align*}
\int_{0}^{T}(B_{\ep}(1_{E}\varphi_{\beta, \ep}(t)), 
                                      1_{E}\varphi_{\beta, \ep}(t))_{H}\,dt 
&= \int_{0}^{T}(B_{\ep}(\varphi_{\beta, \ep}(t)), 
                                      1_{E}\varphi_{\beta, \ep}(t))_{H}\,dt 
\\ 
&\to \int_{0}^{T}(\xi_{\beta}(t), 
                                      1_{E}\varphi_{\beta}(t))_{H}\,dt
\end{align*}
as $\ep=\ep_{j}\searrow0$, 
and hence it holds that 
\begin{align*}
\xi_{\beta} \in B(1_{E}\varphi_{\beta}) 
\quad \mbox{a.e.\ on}\ \Omega\times(0, T)  
\end{align*} 
(see \cite[Lemma 1.3, p.\ 42]{Barbu1}). 
In particular, we see that 
\begin{align*}
\xi_{\beta} \in B(1_{E}\varphi_{\beta}) 
= B(\varphi_{\beta}) 
\quad \mbox{a.e.\ on}\ E\times(0, T).   
\end{align*}
Thus, since $E \subset \Omega$ is arbitrary, 
we can obtain \eqref{xibetainB(varphibeta)}. 
Hence combining \eqref{semidefsolPbe2} and \eqref{xibetainB(varphibeta)}  
leads to \eqref{defsolPbe2}.   

Next we prove \eqref{defsolPbe4}. 
Let $E \subset \Omega$ be 
an arbitrary bounded domain with smooth boundary.  
Then from Lemmas \ref{keylemmaVstar} and \ref{estimatesolPbep1} we have 
\begin{equation}\label{itiyoutilde}
\tilde{J}^{1/2}_{1}\mu_{\beta, \ep} \to 
\tilde{J}^{1/2}_{1}\mu_{\beta} 
\quad \mbox{in}\ C([0, T]; L^2(E))
\end{equation}
as $\ep = \ep_j \searrow0$. 
Thus,   
since $\mu_{\beta, \ep}(0)=\mu_{0\ep} \to \mu_{0}$ in $H$ as $\ep\searrow0$, 
\eqref{itiyoutilde} yields that 
$$
\tilde{J}^{1/2}_{1}\mu_{\beta}(0)=\tilde{J}^{1/2}_{1}\mu_{0} 
\quad \mbox{a.e.\ in}\ E.
$$
Because $E \subset \Omega$ is arbitrary, we conclude that 
$$
\tilde{J}^{1/2}_{1}\mu_{\beta}(0)=\tilde{J}^{1/2}_{1}\mu_{0} 
\quad \mbox{a.e.\ in}\ \Omega.
$$ 
Since $\tilde{J}^{1/2}_{1}\mu_{0} \in H$, we see that 
$$
\tilde{J}^{1/2}_{1}\mu_{\beta}(0)=\tilde{J}^{1/2}_{1}\mu_{0} 
\quad \mbox{in}\ H, 
$$
that is,  
$$
\mu_{\beta}(0)=\mu_{0} 
\quad \mbox{in}\ V^{*}. 
$$ 
Recalling that $\mu_{0} \in H$, we have  
$$
\mu_{\beta}(0)=\mu_{0} 
\quad \mbox{in}\ H. 
$$
Similarly, we can prove that 
$$
\varphi_{\beta}(0)=\varphi_{0}, \quad 
\sigma_{\beta}(0)=\sigma_{0}  
\qquad \mbox{in}\ H. 
$$
Therefore \eqref{defsolPbe4} holds. 

We can show \eqref{betaes1} in a similar way to 
the proof of Lemma \ref{estimatesolPbep1}. 
Indeed, we infer from 
\eqref{defsolPbe1}-\eqref{defsolPbe4} and (C4) that 
there exists a constant $C_{1}>0$ such that 
\begin{align*}
&\frac{\alpha}{2}\|\mu_{\beta}(t)\|_{H}^2 
+ \int_{0}^{t}\|\nabla \mu_{\beta}(s)\|_{H}^2\,ds  \\ 
&+ \beta\int_{0}^{t}\|\partial_{t}\varphi_{\beta}(s)\|_{H}^2\,ds  
+ \frac{1}{2}\|\varphi_{\beta}(t)\|_{V}^2 
+ \int_{\Omega} G(\varphi_{\beta}(t)) \\ 
&+ \frac{1}{2}\|\sigma_{\beta}(t)\|_{H}^2 
+ \int_{0}^{t}\|\nabla \sigma_{\beta}(s)\|_{H}^2\,ds  
+ p\int_{0}^{t}\|\sigma_{\beta}(s)-\mu_{\beta}(s)\|_{H}^2\,ds  \\ 
&\leq C_{1}\Bigl(\alpha\|\mu_0\|_{H}^2 + \|\varphi_0\|_{V}^2 
                     + \|G(\varphi_0)\|_{L^1(\Omega)} 
                     + \|\sigma_0\|_{H}^2 \Bigr),  \\[4mm]
&\frac{1}{2}\|\varphi_{\beta}(t)\|_{V}^2 
+ \int_{\Omega}G(\varphi_{\beta}(t)) 
\\ \notag 
&\geq \frac{1}{2}\|\varphi_{\beta}(t)\|_{V}^2 
         - \frac{\|\pi'\|_{L^{\infty}(\mathbb{R})}}{2}\|\varphi_{\beta}(t)\|_{H}^2 
\\ \notag 
&\geq \frac{1}{2}
               (1-\|\pi'\|_{L^{\infty}(\mathbb{R})})\|\varphi_{\beta}(t)\|_{V}^2, \\[4mm]
&\left|\int_{0}^{T} 
         \langle (\sigma_{\beta})_{t}(t), v(t) \rangle_{V^{*}, V}\,dt \right| \\ 
&\leq (\|\nabla\sigma_{\beta}\|_{L^2(0, T; H)} 
               + p\|\sigma_{\beta}-\mu_{\beta}\|_{L^2(0, T; H)})
                                                                                  \|v\|_{L^2(0, T; V)}, 
\\[4mm] 
&\left|\int_{0}^{T} 
         \langle (\alpha\mu_{\beta}+\varphi_{\beta})_{t}(t), 
                                                                    v(t) \rangle_{V^{*}, V}\,dt \right| 
\\ \notag
&\leq \left|\int_{0}^{T} 
         \langle (\sigma_{\beta})_{t}(t), v(t) \rangle_{V^{*}, V}\,dt \right| 
      + \|\nabla(\mu_{\beta}+\sigma_{\beta})\|_{L^2(0, T; H)}
                                                                             \|v\|_{L^2(0, T; V)} 
\end{align*}
for all $v \in L^2(0, T; V)$, 
and hence \eqref{betaes1} holds.  
  
We can prove \eqref{betaes2} 
in the same way as in the proofs of 
Lemma \ref{estimatesolPbep2}. 
Indeed, we see  
from \eqref{defsolPbe2}, the Young inequality and the elliptic estimate 
that 
\begin{align*}
&\|\varphi_{\beta}\|_{L^2(0, T; W)}^2 
\leq C_{2}(\|\mu_{\beta}\|_{L^2(0, T; H)}^2  
                    + \beta\|\partial_{t}\varphi_{\beta}\|_{L^2(0, T; H)}^2 
                    + \|\pi'\|_{L^{\infty}(\mathbb{R})}^2
                                                   \|\varphi_{\beta}\|_{L^2(0, T; H)}^2), 
\\[2mm] 
&\|\xi_{\beta}\|_{L^2(0, T; H)}  
\leq \|\mu_{\beta}\|_{L^2(0, T; H)} 
       + \beta^{1/2}\|\partial_{t}\varphi_{\beta}\|_{L^2(0, T; H)} 
\\ 
       &\hspace{27mm}+ \|(-\Delta+1)\varphi_{\beta}\|_{L^2(0, T; H)} 
       + \|\pi'\|_{L^{\infty}(\mathbb{R})}\|\varphi_{\beta}\|_{L^2(0, T; H)} 
\end{align*}
with some constant $C_{2}>0$.  
Hence we can obtain \eqref{betaes2} by \eqref{betaes1}.       

We let $(\mu_{j}, \varphi_{j}, \sigma_{j})$, $j=1, 2$, 
be two weak solutions of \ref{Pbeta}, 
and 
put $\overline{\mu}:=\mu_{1}-\mu_{2}, 
\overline{\varphi}:=\varphi_{1}-\varphi_{2}, 
\overline{\sigma}:=\sigma_{1}-\sigma_{2}, 
\overline{\xi}:=\xi_{1}-\xi_{2}, 
\theta:=\alpha\overline{\mu}+\overline{\varphi}+\overline{\sigma}$.  
Then 
we derive from \eqref{defsolPbe2} that 
\begin{align*}
&\int_{0}^{t}\|\overline{\varphi}(s)\|_{V}^2\,ds 
+ \int_{0}^{t}(\overline{\xi}(s), \overline{\varphi}(s))_{H}\,ds 
-\int_{0}^{t} (\overline{\mu}(s), \overline{\varphi}(s))_{H}\,ds 
\\  
&= -\frac{\beta}{2}\|\overline{\varphi}(t)\|_{H}^2 
    -\int_{0}^{t}(\pi(\varphi_{1}(s))-\pi(\varphi_{2}(s)), 
                                                      \overline{\varphi}(s))_{H}\,ds.  
\end{align*}
Thus we can show that 
for all $\alpha>0$ and all $\delta>0$  
there exists a constant $C_{3}=C_{3}(\alpha, \delta)>0$ such that 
\begin{align*}
&\frac{1}{2}\|\theta(t)\|_{V^{*}}^2 
+ \frac{\alpha}{8}\int_{0}^{t}\|\overline{\mu}(s)\|_{H}^2\,ds 
+ (1-2\|\pi'\|_{L^{\infty}(\mathbb{R})}^2\alpha-2\delta)
                                       \int_{0}^{t}\|\overline{\varphi}(s)\|_{V}^2\,ds 
\\ 
&+ \frac{1}{2}\|\overline{\sigma}(t)\|_{H}^2 
  + (1-\delta)\int_{0}^{t}\|\overline{\sigma}(s)\|_{V}^2\,ds 
\\ 
&\leq  C_{3}\int_{0}^{t}\|\overline{\sigma}(s)\|_{H}^2\,ds 
         + C_{3}\int_{0}^{t}\|\theta(s)\|_{V^{*}}^2\,ds
\end{align*}
in the same way as in the proof of Lemma \ref{uniquesolPbep}. 
Therefore, choosing $\delta > 0$ and $\alpha>0$ small enough 
and applying the Gronwall lemma,  
we see that $\theta=\overline{\mu}=\overline{\varphi}=\overline{\sigma}=0$, 
which implies that  
$\mu_{1}=\mu_{2}$, $\varphi_{1}=\varphi_{2}$ and $\sigma_{1}=\sigma_{2}$. 
Then we have $\xi_{1}=\xi_{2}$ by \eqref{defsolPbe2}. 
Hence we can establish uniqueness of weak solutions to \ref{Pbeta}.   \qed
\end{prth1.1}
%
%


\section{Proof of Theorem \ref{maintheorem2}}\label{Sec4} 

This section proves Theorem \ref{maintheorem2}. 
The following lemma asserts Cauchy's criterion 
for solutions of \ref{Pbeta} 
and is the key to the proof of Theorem \ref{maintheorem2}.
\begin{lem}\label{lemCauchy}
Let $(\mu_{\beta}, \varphi_{\beta}, \sigma_{\beta}, \xi_{\beta})$ 
be a weak solution of {\rm \ref{Pbeta}} 
and let $\alpha_{0}$ be as in Theorem \ref{maintheorem1}. 
Then there exists $\alpha_{2} \in (0, \alpha_{0}]$ such that 
 for all $\alpha \in (0, \alpha_{2})$
 there exists a constant $C=C(\alpha, T)>0$ such that 
     \begin{align*}
     &\|\mu_{\beta}-\mu_{\eta}\|_{L^2(0, T; H)} 
     + \|\varphi_{\beta}-\varphi_{\eta}\|_{L^2(0, T; V)} 
     + \|\sigma_{\beta}-\sigma_{\eta}\|_{L^{\infty}(0, T; H) \cap L^2(0, T; V)} \\ \notag  
     &+\|
           (\alpha\mu_{\beta}+\varphi_{\beta}+\sigma_{\beta}) 
           - (\alpha\mu_{\eta}+\varphi_{\eta}+\sigma_{\eta})
         \|_{L^{\infty}(0, T; V^{*})} \\ \notag 
     &\leq C(\beta^{1/2} + \eta^{1/2})  
     \end{align*}
for all $\beta\in (0, 1)$ and all $\eta \in (0, 1)$.    
\end{lem}
\begin{proof}
We let $\beta, \eta \in (0, 1)$ and  
put $\overline{\mu}:=\mu_{\beta}-\mu_{\eta}, 
\overline{\varphi}:=\varphi_{\beta}-\varphi_{\eta}, 
\overline{\sigma}:=\sigma_{\beta}-\sigma_{\eta}, 
\overline{\xi}:=\xi_{\beta}-\xi_{\eta}, 
\theta:=\alpha\overline{\mu}+\overline{\varphi}+\overline{\sigma}$.  
Then, since we see from \eqref{defsolPbe2} that 
\begin{align*}
&\int_{0}^{t}\|\overline{\varphi}(s)\|_{V}^2\,ds 
+ \int_{0}^{t}(\overline{\xi}(s), \overline{\varphi}(s))_{H}\,ds 
-\int_{0}^{t} (\overline{\mu}(s), \overline{\varphi}(s))_{H}\,ds 
\\  
&=  -\int_{0}^{t}
       (\beta\partial_{t}\varphi_{\beta}(s)-\eta\partial_{t}\varphi_{\eta}(s), 
                                                                      \overline{\varphi}(s))_{H}\,ds  
    -\int_{0}^{t}(\pi(\varphi_{\beta}(s))-\pi(\varphi_{\eta}(s)), 
                                                      \overline{\varphi}(s))_{H}\,ds,   
\end{align*}
we can prove that 
for all $\alpha>0$ and all $\delta>0$  
there exists a constant $C_{1}=C_{1}(\alpha, \delta)>0$ such that 
\begin{align}\label{jojo1}
&\frac{1}{2}\|\theta(t)\|_{V^{*}}^2 
+ \frac{\alpha}{8}\int_{0}^{t}\|\overline{\mu}(s)\|_{H}^2\,ds 
+ (1-2\|\pi'\|_{L^{\infty}(\mathbb{R})}^2\alpha-2\delta)
                                       \int_{0}^{t}\|\overline{\varphi}(s)\|_{V}^2\,ds 
\\ \notag
&+ \frac{1}{2}\|\overline{\sigma}(t)\|_{H}^2 
  + (1-\delta)\int_{0}^{t}\|\overline{\sigma}(s)\|_{V}^2\,ds 
\\ \notag
&\leq 
         -\int_{0}^{t}
       (\beta\partial_{t}\varphi_{\beta}(s)-\eta\partial_{t}\varphi_{\eta}(s), 
                                                                      \overline{\varphi}(s))_{H}\,ds  
         + C_{1}\int_{0}^{t}\|\overline{\sigma}(s)\|_{H}^2\,ds 
         + C_{1}\int_{0}^{t}\|\theta(s)\|_{V^{*}}^2\,ds
\end{align}
in the same way as in the proof of Lemma \ref{uniquesolPbep}. 
Here, from \eqref{betaes1} the Young inequality yields that  
for all $\delta>0$ there exists a constant $C_{2}=C_{2}(\delta, T)>0$ such that 
\begin{align}\label{jojo2}
-\int_{0}^{t}
       (\beta\partial_{t}\varphi_{\beta}(s)-\eta\partial_{t}\varphi_{\eta}(s), 
                                                                      \overline{\varphi}(s))_{H}\,ds  
\leq \delta\int_{0}^{t}\|\overline{\varphi}(s)\|_{V}^2\,ds.  
       + C_{2}(\beta + \eta). 
\end{align} 
Thus, by virtue of \eqref{jojo1} and \eqref{jojo2}, we have 
\begin{align*} 
&\frac{1}{2}\|\theta(t)\|_{V^{*}}^2 
+ \frac{\alpha}{8}\int_{0}^{t}\|\overline{\mu}(s)\|_{H}^2\,ds 
+ (1-2\|\pi'\|_{L^{\infty}(\mathbb{R})}^2\alpha-3\delta)
                                       \int_{0}^{t}\|\overline{\varphi}(s)\|_{V}^2\,ds 
\\ \notag
&+ \frac{1}{2}\|\overline{\sigma}(t)\|_{H}^2 
  + (1-\delta)\int_{0}^{t}\|\overline{\sigma}(s)\|_{V}^2\,ds 
\\ \notag
&\leq  C_{1}\int_{0}^{t}\|\overline{\sigma}(s)\|_{H}^2\,ds 
         + C_{1}\int_{0}^{t}\|\theta(s)\|_{V^{*}}^2\,ds 
         + C_{2}(\beta + \eta). 
\end{align*} 
Hence, choosing $\delta > 0$ and $\alpha>0$ small enough 
and applying the Gronwall lemma, 
we can obtain Lemma \ref{lemCauchy}.  
\end{proof}
%
%
\begin{prth1.2} 
Let $\alpha_{2}$ be as in Lemma \ref{lemCauchy} 
and let $\alpha \in (0, \alpha_{2})$. 
Then, by \eqref{betaes1}, \eqref{betaes2} and Lemma \ref{lemCauchy}, 
there exist some functions 
\begin{align*}
&\mu \in L^{\infty}(0, T; H) \cap L^2(0, T; V), 
    \\
    &\varphi \in L^{\infty}(0, T; V) \cap L^2(0, T; W), 
    \\
    &\sigma \in H^1(0, T; V^{*}) \cap L^2(0, T; V), 
    \\
    &\xi \in L^2(0, T; H)
\end{align*}
such that 
$\alpha\mu+\varphi \in H^1(0, T; V^{*})$   
and 
\begin{align*}
&\mu_{\beta} \to \mu 
\quad \mbox{weakly$^{*}$ in}\ L^{\infty}(0, T; H)\ 
\mbox{and weakly$^{*}$ in}\ L^2(0, T; V), 
\\ 
&\mu_{\beta} \to \mu 
\quad \mbox{in}\ L^{2}(0, T; H), 
\\ 
&\varphi_{\beta} \to \varphi 
\quad \mbox{weakly$^{*}$ in}\ L^{\infty}(0, T; V), 
\\
&\varphi_{\beta} \to \varphi 
\quad \mbox{in}\ L^{2}(0, T; V), 
\\
&\sigma_{\beta} \to \sigma
\quad \mbox{in}\ L^{\infty}(0, T; H) \cap L^2(0, T; V), 
\\
&(\alpha\mu_{\beta}+\varphi_{\beta})_{t} 
\to (\alpha\mu+\varphi)_{t} 
\quad \mbox{weakly in}\ L^2(0, T; V^{*}), 
\\
&(\alpha\mu_{\beta}+\varphi_{\beta}+\sigma_{\beta})_{t} 
\to (\alpha\mu+\varphi+\sigma)_{t} 
\quad \mbox{in}\ L^{\infty}(0, T; V^{*}), 
\\
&\beta\partial_{t}\varphi_{\beta} \to 0
\quad \mbox{in}\ L^2(0, T; H), 
\\
&(\sigma_{\beta})_{t} \to \sigma_{t}  
\quad \mbox{weakly in}\ L^2(0, T; V^{*}), 
\\
&\xi_{\beta} \to \xi 
\quad \mbox{weakly in}\ L^2(0, T; H), 
\\
&(-\Delta+1)\varphi_{\beta}  \to (-\Delta+1)\varphi  
\quad \mbox{weakly in}\ L^2(0, T; H) 
\end{align*}
as $\beta=\beta_{j}\searrow0$.  
Hence we can show \eqref{defsolP1}-\eqref{defsolP3} and \eqref{eres}. 
Moreover, we can prove \eqref{defsolP4} 
by \eqref{betaes1}, \eqref{betaes2} and Lemma \ref{keylemmaVstar}. 
Thus we can establish existence of weak solutions to \eqref{P}.     

Next we prove uniqueness of weak solutions to \eqref{P}. 
We let $(\mu_{j}, \varphi_{j}, \sigma_{j}, \xi_{j})$, $j=1, 2$, 
be two weak solutions of \eqref{P}, 
and 
put $\overline{\mu}:=\mu_{1}-\mu_{2}, 
\overline{\varphi}:=\varphi_{1}-\varphi_{2}, 
\overline{\sigma}:=\sigma_{1}-\sigma_{2}, 
\overline{\xi}:=\xi_{1}-\xi_{2}, 
\theta:=\alpha\overline{\mu}+\overline{\varphi}+\overline{\sigma}$.  
Then we infer from \eqref{defsolP2} that 
\begin{align*}
&\int_{0}^{t}\|\overline{\varphi}(s)\|_{V}^2\,ds 
+ \int_{0}^{t}(\overline{\xi}(s), \overline{\varphi}(s))_{H}\,ds 
-\int_{0}^{t} (\overline{\mu}(s), \overline{\varphi}(s))_{H}\,ds 
\\  
&= -\int_{0}^{t}(\pi(\varphi_{1}(s))-\pi(\varphi_{2}(s)), 
                                                      \overline{\varphi}(s))_{H}\,ds, 
\end{align*}
and hence we can show that 
for all $\alpha>0$ and all $\delta>0$  
there exists a constant $C_{1}=C_{1}(\alpha, \delta)>0$ such that 
\begin{align*}
&\frac{1}{2}\|\theta(t)\|_{V^{*}}^2 
+ \frac{\alpha}{8}\int_{0}^{t}\|\overline{\mu}(s)\|_{H}^2\,ds 
+ (1-2\|\pi'\|_{L^{\infty}(\mathbb{R})}^2\alpha-2\delta)
                                       \int_{0}^{t}\|\overline{\varphi}(s)\|_{V}^2\,ds 
\\ 
&+ \frac{1}{2}\|\overline{\sigma}(t)\|_{H}^2 
  + (1-\delta)\int_{0}^{t}\|\overline{\sigma}(s)\|_{V}^2\,ds 
\\ 
&\leq  C_{1}\int_{0}^{t}\|\overline{\sigma}(s)\|_{H}^2\,ds 
         + C_{1}\int_{0}^{t}\|\theta(s)\|_{V^{*}}^2\,ds
\end{align*}
in the same way as in the proof of Lemma \ref{uniquesolPbep}. 
Hence, choosing $\delta > 0$ and $\alpha>0$ small enough 
and applying the Gronwall lemma yield that 
$\theta=\overline{\mu}=\overline{\varphi}=\overline{\sigma}=0$, 
which implies that  
$\mu_{1}=\mu_{2}$, $\varphi_{1}=\varphi_{2}$ and $\sigma_{1}=\sigma_{2}$. 
Then we have $\xi_{1}=\xi_{2}$ by \eqref{defsolP2}.  

Therefore we can obtain Theorem \ref{maintheorem2}. 
\qed
\end{prth1.2}

%
 
\end{document}